\documentclass[10pt]{article}
\usepackage[utf8]{inputenc}
\usepackage[english]{babel}
\usepackage{amsmath,amsthm,amssymb}
\usepackage{stmaryrd}
\usepackage{enumitem}
\usepackage{caption}
\usepackage{authblk}
\usepackage{xcolor,tikz}
\usetikzlibrary{positioning,arrows,trees,decorations.text}
\usepackage{arydshln}
\usepackage[all,cmtip]{xy}
\usepackage[symbol, bottom]{footmisc}
\usepackage{hyperref}
\hypersetup{colorlinks=true, linkcolor=blue}

\newtheorem{theorem}{Theorem}[section]

\newtheorem{corollary}[theorem]{Corollary}
\newtheorem{lemma}[theorem]{Lemma}
\newtheorem{proposition}[theorem]{Proposition}

\theoremstyle{definition}
\newtheorem{definition}[theorem]{Definition}

\theoremstyle{remark} \theoremstyle{remark}
\newtheorem{remark}[theorem]{Remark}
\newtheorem{example}[theorem]{Example}

\numberwithin{equation}{section}

\newcommand{\R}{\mathbb{R}}
\newcommand{\C}{\mathbb{C}}
\newcommand{\X}{\mathfrak{X}}
\newcommand{\D}{\mathcal{D}}
\newcommand{\A}{\mathcal{A}}
\newcommand{\f}{\varphi}
\newcommand{\Lie}{\mathcal{L}}
\newcommand{\calC}{\mathcal{C}}
\newcommand{\Ric}{\mathrm{Ric}}

\makeatletter
\def\@fnsymbol#1{\ensuremath{\ifcase#1\or 1\or 2\or \ddagger\or \mathsection\or \mathparagraph\or \|\or \#\or \dagger\dagger\or \ddagger\ddagger \else\@ctrerr\fi}}

\begin{document}
	\title{\textbf{Anti-quasi-Sasakian manifolds}}
	\author{D. Di Pinto\thanks{Dario Di Pinto (corresponding author): {\tt dario.dipinto@uniba.it}} $\,$ and G. Dileo\thanks{Giulia Dileo: {\tt giulia.dileo@uniba.it} 
	\smallskip
			
	$^{1,2}$ Dipartimento di Matematica, Università degli Studi di Bari Aldo Moro, Via E. Orabona 4, 70125 Bari, Italy.}}
	\date{}	
	\maketitle

\begin{abstract}
	We introduce and study a special class of almost contact metric manifolds, which we call anti-quasi-Sasakian (aqS). Among the class of transversely K\"ahler almost contact metric manifolds $(M,\varphi, \xi,\eta,g)$, quasi-Sasakian and anti-quasi-Sasakian manifolds are characterized, respectively, by the $\varphi$-invariance and the $\varphi$-anti-invariance of the $2$-form $d\eta$. A Boothby-Wang type theorem allows to obtain aqS structures on principal circle bundles over K\"ahler manifolds endowed with a closed $(2,0)$-form. We characterize aqS manifolds with constant $\xi$-sectional curvature equal to $1$: they admit an $Sp(n)\times 1$-reduction of the frame bundle such that the manifold is transversely hyperk\"ahler, carrying a second aqS structure and a null Sasakian $\eta$-Einstein structure. We show that aqS manifolds with constant sectional curvature are necessarily flat and cok\"ahler. Finally, by using a metric connection with torsion, we provide a sufficient condition for an aqS manifold to be locally decomposable as the Riemannian product of a K\"ahler manifold and an aqS manifold with structure of maximal rank. 
	Under the same hypothesis, $(M,g)$ cannot be locally symmetric.
\end{abstract}
\bigskip

{\noindent \small \textbf{MSC (2020):} Primary: 53C15, 53D15, 53C25. Secondary: 53C10, 53C26, 53B05.
\smallskip

\noindent	\textbf{Keywords and phrases:} quasi-Sasakian manifold, anti-quasi-Sasakian manifold, transversely K\"ahler structure, $\mathbb{S}^1$-bundle, complex $(2,0)$-form, hyperk\"ahler structure, $\eta$-Einstein, connection with torsion.
}

\tableofcontents

\section*{Introduction}
\addcontentsline{toc}{section}{Introduction}

\hspace{\parindent} Quasi-Sasakian manifolds are a special class of almost contact metric manifolds first introduced by D. E. Blair in \cite{Blair-qS}, and afterwards studied by various authors \cite{Tanno,Kanemaki1,Kanemaki2,Olszak-qS,CM.dN.Y.}. They are normal almost contact metric manifolds $(M,\varphi,\xi,\eta,g)$, whose fundamental $2$-form $\Phi$, defined by $\Phi(X,Y)=g(X,\varphi Y)$, is closed. They include both cok\"ahler and Sasakian manifolds, which satisfy $d\eta=0$ and $d\eta=2\Phi$ respectively. In fact in these two cases the $1$-form $\eta$ attains the minimum and the maximum rank in the sense of the definition given in \cite{Blair-qS}. The normality condition of the structure expresses the integrability of an almost complex structure $J$ defined on the product manifold $M\times\mathbb{R}$: this is equivalent to the vanishing of the tensor field $N_\varphi=[\varphi,\varphi]+d\eta\otimes\xi$, where $[\varphi,\varphi]$ denotes the Nijenhuis torsion of $\varphi$. The Reeb vector field $\xi$ of any quasi-Sasakian manifold is Killing. This, together with normality and $d\Phi=0$, ensures that the structure $(\varphi,g)$ is projectable along the $1$-dimensional foliation generated by $\xi$ and the transverse geometry is K\"ahler (see \cite{CM.dN.Y.}).

In the present paper we want to introduce a new class of almost contact metric manifolds $(M,\varphi,\xi,\eta,g)$ with projectable structure $(\varphi,g)$ and such that the transverse geometry with respect to $\xi$ is given by a K\"ahler structure endowed with a closed $(2,0)$-form. This even dimensional geometry is of particular interest: hyperk\"ahler manifolds are well known examples of K\"ahler manifolds with a (nondegenerate) closed $(2,0)$-form. In general, the nondegeneracy of the $(2,0)$-form forces the dimension of the manifold to be multiple of $4$. By a result of A. Beauville, compact K\"ahler manifolds endowed with a nondegenerate closed $(2,0)$-form (also called complex symplectic) admit a hyperkähler structure (cfr. \cite{Beauville} and \cite[14.B]{Besse}).

In order to define the new class of almost contact metric manifolds, which will be called \textit{anti-quasi-Sasakian manifolds} (aqS manifolds for short), we need to modify the normality condition to an \textit{anti-normal} condition, in such a way that we loose the integrability of the structure $J$ defined on $M\times \mathbb{R}$, but not the projectability of $\varphi$ along $\xi$ and the integrability of the induced transverse almost complex structure. Precisely, we define an anti-quasi-Sasakian manifold as an almost contact metric manifold such that
\[d\Phi=0,\qquad  N_\varphi=2d\eta\otimes\xi.\]
The Reeb vector field $\xi$ is again Killing. It is easily seen that the manifold is both quasi-Sasakian and anti-quasi-Sasakian if and only if it is cok\"ahler. This is coherent with the placing of these manifolds in the Chinea-Gonzalez classification of almost contact metric manifolds.
Indeed, it is known that quasi-Sasakian manifolds coincide with the class $\mathcal{C}_6\oplus\mathcal{C}_7$, and we show that anti-quasi-Sasakian manifolds coincide with manifolds in the class $\mathcal{C}_{10}\oplus\mathcal{C}_{11}$ satisfying the additional condition $\mathcal{L}_\xi\varphi=0$, where $\mathcal{L}_\xi$ denotes the Lie derivative with respect to $\xi$. We characterize manifolds in the class $\mathcal{C}_6\oplus\mathcal{C}_7\oplus\mathcal{C}_{10}\oplus\mathcal{C}_{11}$, calling them \textit{generalized quasi-Sasakian manifolds}, and showing that they are exactly all \textit{transversely K\"ahler} almost contact metric manifolds, provided that $\mathcal{L}_\xi\varphi=0$.

Among the class of transversely K\"ahler almost contact metric manifolds, quasi-Sasakian and anti-quasi-Sasakian manifolds are characterized respectively by
\[d\eta(\varphi X,\varphi Y)=d\eta(X,Y),\qquad d\eta(\varphi X,\varphi Y)=-d\eta(X,Y),\]
that is the $\varphi$-invariance and the $\varphi$-anti-invariance of $d\eta$, which justifies the name for the new class (see Figure \ref{Fig:1}). In fact, for an anti-quasi-Sasakian manifold $(M,\varphi,\xi,\eta,g)$, considering the local Riemannian submersion $\pi:M\to M/\xi$ onto a K\"ahler manifold, $d\eta$ is the projectable $2$-form which induces a closed $(2,0)$-form $\omega$ on the K\"ahler base space. Under the hypothesis that the Reeb vector field is regular with compact orbits, a Boothby-Wang type theorem holds, namely, $M$ is a principal circle bundle over $M/\xi$ and $\eta$ is a connection form, whose curvature is $d\eta=\pi^*\omega$ (Theorem \ref{Thm:Boothby-Wang1}). Conversely, considering a K\"ahler manifold endowed with a closed $(2,0)$-form $\omega$ defining an integral cohomology class, we show that there exists a principal $\mathbb{S}^1$-bundle $M$ endowed with an anti-quasi-Sasakian structure $(\varphi,\xi,\eta,g)$ such that $\eta$ is a connection form with curvature $d\eta=\pi^*\omega$ (Theorem \ref{Thm:Boothby-Wang2}).

If the $1$-form $\eta$ of an aqS manifold is a contact form, or equivalently the transverse $(2,0)$-form is nondegenerate, owing to the $\varphi$-anti-invariance of $d\eta$, the dimension of $M$ turns out to be of type $4n+1$. In the general case, one can consider the $\varphi$-invariant distribution defined by $\mathcal{E}=\ker\eta\cap\ker (d\eta)$. If it is of constant rank $2q$, then $\dim M=2q+4p+1$, where $\eta\wedge (d\eta)^{2p}\ne0$ and $d\eta^{2p+1}=0$: we say that the aqS structure has rank $4p+1$. Obviously cok\"ahler manifolds are aqS manifolds of minimal rank $1$.

Examples of anti-quasi-Sasakian manifolds are discussed in Section \ref{Sec:examples}. Compact nilmanifolds endowed with an aqS structure can be obtained as quotients of a $2$-step nilpotent Lie group, which we call weighted Heisenberg Lie group, with structure constants depending on some weights $\lambda_1,\ldots,\lambda_n$. This Lie group (which is transversely flat), actually carries two anti-quasi-Sasakian structures and one quasi-Sasakian structure; the latter coincides with the well known Sasakian structure of the Heisenberg Lie group when the weights are all equal to $1$. Compact aqS manifolds which are not quotients of the weighted Heisenberg Lie group can be obtained applying the $\mathbb{S}^1$-bundle construction to non-flat compact hyperk\"ahler manifolds with integral $2$-forms (see Example \ref{ex:compact-Cortes}).

Anti-quasi-Sasakian structures of maximal rank arise naturally on a special class of Riemannian manifolds $(M,g)$, for which the structure group of the frame bundle is reducible to $Sp(n)\times1$. Such a reduction is equivalent to the existence of three compatible almost contact structures $(\varphi_i,\xi,\eta)$, $i=1,2,3$, sharing the same Reeb vector field, and satisfying the quaternionic identities $\varphi_i\varphi_j=\varphi_k=-\varphi_j\varphi_i$ for every even permutation $(i,j,k)$ of $(1,2,3)$. If the fundamental $2$-forms satisfy
\[d\Phi_1=0,\qquad d\Phi_2=0,\qquad d\eta=2\Phi_3,\]
we show that the first two structures are anti-normal, and thus anti-quasi-Sasakian, and the third one is normal, and thus Sasakian. We say that $(\f_i,\xi,\eta,g)$ ($i=1,2,3$) is a \textit{double aqS-Sasakian structure}. For $n=1$, it is a special kind of $K$-contact hypo $SU(2)$-structure \cite{dAFFU}.
Double aqS-Sasakian manifolds are transversely hyperk\"ahler along the Reeb foliation, and hence transversely Ricci-flat. In particular, $(\varphi_3,\xi,\eta,g)$ is a null Sasakian $\eta$-Einstein structure \cite{BGM}. It is well known that for any Sasakian manifold, the $\xi$-sectional curvatures, that is, sectional curvatures of $2$-planes containing $\xi$, are all equal to $1$. We will show that the class of double aqS-Sasakian manifolds provides all anti-quasi-Sasakian manifolds of constant $\xi$-sectional curvature $K(\xi,X)=1$.

In general, every anti-quasi-Sasakian manifold $(M,\varphi,\xi,\eta,g)$ admits a triplet of $2$-forms $(\mathcal{A},\Phi,\Psi)$ satisfying
\[d\mathcal{A}=0,\qquad d\Phi=0,\qquad d\eta=2\Psi.\]
Beside $\Phi$, which is the fundamental $2$-form of the structure, $\mathcal{A}$ and $\Psi$ are the $2$-forms associated to the skew-symmetric operators $A:=-\varphi\circ\nabla\xi$ and $\psi:=A\varphi$, both anticommuting with $\varphi$. Here $\nabla$ denotes the Levi-Civita connection of the metric $g$. The spectrum of the symmetric operator $\psi^2=A^2$ encodes information on the Riemannian geometry of the manifold. The requirement for $\psi$ and $A$ to be both almost contact structures, namely $\operatorname{Sp}(\psi^2)=\{0,-1\}$ with $0$ simple eigenvalue, is equivalent to constant $\xi$-sectional curvature $1$, in which case $(A,\varphi,\psi,\xi,\eta,g)$ is a double aqS-Sasakian structure (Theorem \ref{Thm:K(X,xi)=1}).

We will see that for an aqS manifold the condition to be of constant sectional curvature forces the manifold to be flat and cok\"ahler (Theorem \ref{Thm:constant curvature}). This is consequence of general properties of the Riemannian Ricci curvature.
In this regard, assuming the transverse K\"ahler structure of an aqS manifold $(M,\varphi,\xi,\eta,g)$ to be Einstein, we show that $M$ is $\eta$-Einstein if and only if $\operatorname{Sp}(\psi^2)=\{0,-\lambda^2\}$, with $0$ simple eigenvalue and $\lambda$ constant, in which case $M$ is necessarily transversely Ricci-flat (Theorem \ref{Prop.eta-Einstein}). An example of transversely K\"ahler-Einstein aqS manifold which is not transversely Ricci-flat, can be obtained applying the $\mathbb{S}^1$-bundle construction to the complex unit disc $D^2$ with constant holomorphic sectional curvature $c<0$. In this case, the unique nonvanishing eigenvalue of the operator $\psi^2$ is non constant (see Example \ref{example-disc}). Other obstructions to the existence of aqS structures are discussed in the compact and homogeneous cases.

Finally, in the last section of the paper we consider a metric connection $\bar\nabla$ with nonvanishing torsion, adapted to the aqS structure. We prove that the parallelism of $\psi$ with respect to $\bar\nabla$, provides a sufficient condition for the manifold to be locally decomposable as the Riemannian product of a K\"ahler manifold and an aqS manifold of maximal rank. We also prove that, in the non cok\"ahler case, under the hypothesis $\bar\nabla\psi=0$, the Riemannian manifold $(M,g)$ cannot be locally symmetric.

\section{The class of anti-quasi-Sasakian manifolds}
\subsection{Review of almost contact structures}

An \textit{almost contact manifold} is an odd dimensional smooth manifold $M^{2n+1}$ endowed with a $(1,1)$-tensor field $\f$, a vector field $\xi$, called Reeb vector field, and a $1$-form $\eta$ satisfying
$$\f^2=-I+\eta\otimes\xi,\quad \eta(\xi)=1,$$
which imply $\f(\xi)=0$ and $\eta\circ\f=0$. The tangent bundle splits as $TM=\D\oplus\langle\xi\rangle$, where $\langle\xi\rangle=\R\xi$ and $\D:=\ker\eta=\operatorname{Im}\f$ is a hyperplane distribution. In particular $\f^2=-I$ on $\D$. The two distributions $\langle\xi\rangle$ and $\D$ are called \textit{vertical} and \textit{horizontal}, respectively. We will denote by $\Gamma(\D)$ the module of smooth sections of $\D$.
Given an almost contact manifold $(M,\f,\xi,\eta)$, on the product manifold $M\times\R$ one can define an almost complex structure $J$  by
\begin{equation}\label{eq:J_MxR}
	J\Big(X,a\frac{d}{dt}\Big)=\Big(\f X-a\xi,\eta(X)\frac{d}{dt}\Big),
\end{equation}
where $X\in\X(M)$ and $a$ is a differentiable function on $M\times\R$.
$(M,\f,\xi,\eta)$ is called \textit{normal} if $J$ is integrable; this is equivalent to the vanishing of the tensor field $N_\f:=[\f,\f]+d\eta\otimes\xi$, explicitely given by
$$N_\f(X,Y)=[\f X,\f Y]+\f^2[X,Y]-\f[X,\f Y]-\f[\f X,Y]+d\eta(X,Y)\xi$$
for every $X,Y\in\X(M)$. Throughout the paper we will use the following convention for the differential of a $1$-form: $d\eta(X,Y)=X(\eta(Y))-Y(\eta(X))-\eta([X,Y])$.

A Riemannian metric $g$ is \textit{compatible} with the almost contact structure $(\f,\xi,\eta)$ if $g(\f X,\f Y)=g(X,Y)-\eta(X)\eta(Y)$ for every $X,Y\in\X(M)$. With respect to such a metric, $\xi$ is a unit vector field orthogonal to $\D$ and $(\f,\xi,\eta,g)$ is called an \textit{almost contact metric structure} on $M$. The \textit{fundamental $2$-form} associated to the structure is defined by $\Phi(X,Y)=g(X,\f Y)$.
An almost contact metric manifold is called
\begin{itemize}[noitemsep]
	\item[-] \textit{cok\"ahler} if $N_\f=0$, $d\eta=0$, $d\Phi=0$;
	\item[-] \textit{Sasakian} if $N_\f=0$, $d\eta=2\Phi$;
	\item[-] \textit{quasi-Sasakian} if $N_\f=0$, $d\Phi=0$.
\end{itemize}
The condition $d\eta=2\Phi$ defines \textit{contact metric} structures. If in addition $\xi$ is Killing, $M$ is called a \textit{$K$-contact} manifold. In general, the Reeb vector field of any quasi-Sasakian manifold is Killing.

The Levi-Civita connection of an almost contact metric manifold can be expressed by means of the following equation (\cite[Lemma 6.1]{Blair}):
\begin{eqnarray}\label{eq:g(nablaf,.)}
	2g((\nabla_X\f)Y,Z)\nonumber
	&=&d\Phi(X,\f Y,\f Z)-d\Phi(X,Y,Z)+g(N_\f(Y,Z),\f X)\nonumber\\
	&&{}+d\eta(\f Y,Z)\eta(X)-d\eta(\f Z,Y)\eta(X)\\
	&&{}+d\eta(\f Y,X)\eta(Z)-d\eta(\f Z,X)\eta(Y).\nonumber
\end{eqnarray}
In particular, cok\"ahler and Sasakian manifolds are respectively characterized by
$$(\nabla_X\f)Y=0,\qquad (\nabla_X\f)Y=g(X,Y)\xi-\eta(Y)X\qquad \forall X,Y\in\X(M).$$
Finally we recall a technical fact which will be used various times in the following.
\smallskip

\begin{remark}\label{rmk:N(xi,.)-Lie_f}
	In any almost contact manifold $(M,\f,\xi,\eta)$ the condition  $N_\f(\xi,\cdot)=0$ is equivalent to $\Lie_\xi\f=0$, in which case $d\eta(\xi,\cdot)=0$. This is consequence of the following two identities, both obtained by direct computations:
	\begin{equation}\label{eq:N(xi,.)}
		N_\f(\xi,X)=-\f(\Lie_\xi\f)X+d\eta(\xi,X)\xi,
	\end{equation}
	\begin{equation}\label{eq:d_eta(xi,.)}
		d\eta(\xi,X)=\eta((\Lie_\xi\f)\f X),
	\end{equation}
 for every $X\in\X(M)$.
\end{remark}

\subsection{Anti-normal almost contact structures}
The notion of anti-quasi-Sasakian manifolds will involve a condition on the tensor field $N_\f$ for which the integrability of the structure $J$ defined in \eqref{eq:J_MxR} will depend on the rank of the $1$-form $\eta$.
\begin{definition}
We say that an almost contact manifold $(M,\f,\xi,\eta)$ is \textit{anti-normal} if
\begin{equation}\label{eq:anti-normal_new}
	N_\f(X,Y)=2d\eta(X,Y)\xi\quad \forall X,Y\in\X(M).
\end{equation}
It turns out that if the structure is both normal and anti-normal, then $\eta$ is closed, that is, the distribution $\D$ is integrable.
\end{definition}

\begin{proposition}\label{prop-deta-normal}
	 For an anti-normal almost contact manifold $(M,\f,\xi,\eta)$ the following hold:
	\begin{enumerate}[label=(\roman*)]
        \item $d\eta(\xi,\cdot)=0$, or equivalently $\Lie_\xi\eta=0$;
        \item $\Lie_\xi d\eta=0$;
		\item $\Lie_\xi\f=0$;
		\item $d\eta(\f X,\f Y)=-d\eta(X,Y)$, or equivalently $d\eta(\f X,Y)=d\eta(X,\f Y)$ for every $X,Y\in\X(M)$.		
	\end{enumerate}
\end{proposition}
\begin{proof}

Comparing \eqref{eq:N(xi,.)} and \eqref{eq:anti-normal_new}, for every $X\in\X(M)$ we have $-\f(\Lie_\xi\f)X=d\eta(\xi,X)\xi$, which both vanish, being the left hand side horizontal and the right hand side vertical. Using also \eqref{eq:d_eta(xi,.)}, $\Lie_\xi\f=0$, so that (i) and (iii) are proved.
Since the exterior differential commutes with the Lie derivative, by (i) we also have $\Lie_\xi d\eta=0$.
As regards (iv), by the definition of $N_\f$, for every $X,Y\in\X(M)$ one has
\begin{equation}\label{eq:eta(N_f)}
	\eta(N_\f(X,Y))=-d\eta(\f X,\f Y)+d\eta(X,Y).
\end{equation}
Thus, by \eqref{eq:anti-normal_new} we get $d\eta(\f X,\f Y)=-d\eta(X,Y)$ for every $X,Y\in\X(M)$, which is equivalent to $d\eta(\f X,Y)= d\eta(X,\f Y)$ because $d\eta(\xi,\cdot)=0$.
\end{proof}
\smallskip

\begin{remark}
	The defining condition \eqref{eq:anti-normal_new} of an anti-normal structure is equivalent to
	\begin{equation}\label{eq:deta-normal}
		N_\f(\xi,X)=0,\quad N_\f(X,Y)=2d\eta(X,Y)\xi\quad\forall X,Y\in\Gamma(\D).
	\end{equation}
	Indeed, if the structure is anti-normal, from the above proposition $d\eta(\xi,\cdot)=0$, so that $N_\f(\xi,\cdot)=0$. Conversely, assuming \eqref{eq:deta-normal}, by Remark \ref{rmk:N(xi,.)-Lie_f}, again $d\eta(\xi,\cdot)=0$, and hence \eqref{eq:anti-normal_new} holds.
\end{remark}
\medskip

\begin{remark}\label{Rmk:aqS_bis}
	Notice that for any almost contact structure $(\f,\xi,\eta)$, equation \eqref{eq:eta(N_f)} implies:
	\begin{itemize}[noitemsep]
		\item[-] $d\eta$ is $\f$-invariant if and only if $\eta(N_\f(X,Y))=0$, as in the normal case;\smallskip
		\item[-] $d\eta$ is $\f$-anti-invariant if and only if $\eta(N_\f(X,Y))=2d\eta(X,Y)$, as in the anti-normal case.
	\end{itemize}
	In particular, referring to horizontal vector fields, the second equation in \eqref{eq:deta-normal}, is equivalent to $$\eta(N_\f(X,Y))=2d\eta(X,Y),\quad \f(N_\f(X,Y))=0 \quad \forall X,Y\in\Gamma(\D),$$
	and thus to
	$$d\eta(\f X,\f Y)=-d\eta(X,Y),\quad N_\f(X,Y)_\D=0 \quad \forall X,Y\in\Gamma(\D),$$  where $N_\f(X,Y)_\D$ denotes the component along $\D$.
\end{remark}
\medskip

Given an almost contact manifold $(M,\f,\xi,\eta)$, according to the definition of D. E. Blair \cite{Blair-qS}, we say that the 1-form $\eta$ has constant rank $2r$ if $(d\eta)^r\neq0$ and $\eta\wedge(d\eta)^r=0$, and it has constant rank $2r+1$ if $\eta\wedge(d\eta)^r\neq0$ and $(d\eta)^{r+1}=0$. We also say that this is the rank of the almost contact structure $(\f,\xi,\eta)$.
As in the quasi-Sasakian case, the condition $d\eta(\xi,\cdot)=0$ implies that the rank of $\eta$ cannot be even. For an anti-quasi-Sasakian manifold we will prove that $\operatorname{rk}(\eta)=4p+1$. To this aim we prove the following:
\smallskip

\begin{lemma}
	Let $V$ be a finite dimensional real vector space endowed with a complex structure $J$ and a $2$-form $\omega\neq0$ such that
	\begin{equation}\label{eq:omega_J}
		\omega(JX,Y)=\omega(X,JY).
	\end{equation}
	Then there exists a basis $\{u_h,Ju_h,e_k,f_k,Je_k,Jf_k\}$, $h=1,\dots,q$, $k=1,\dots,p$, with respect to which the matrix of $\omega$ is
	$$\begin{pmatrix}
		0_{2q}&0&0&0&0\\
		0&0&I_p&0&0\\
		0&-I_p&0&0&0\\
		0&0&0&0&-I_p\\
		0&0&0&I_p&0
	\end{pmatrix}$$
	In particular $\dim V=2q+4p$, taking $q=0$ if $\omega$ is nondegenerate.
\end{lemma}
\begin{proof}
	Let $U:=\{u\in V\ |\ \omega(u,\cdot)=0\}$. If $\omega$ is nondegenerate, $U$ is trivial. Otherwise, being $JU=U$ by \eqref{eq:omega_J}, $U$ admits a basis $\{u_h,Ju_h\}$, $h=1,\dots,q$.
	Now let $W$ be a complementary space to $U$ in $V$: since $\omega$ is nondegenerate on $W$, there exist $e_1,f_1\in W$ such that $\omega(e_1,f_1)=1=-\omega(Je_1,Jf_1)$. By the skew-symmetry of $\omega$ and \eqref{eq:omega_J} one has:
	$$\omega(e_1,e_1)=\omega(f_1,f_1)=0,\quad \omega(e_1,Je_1)=\omega(f_1,Jf_1)=0.$$
	Putting $k=\omega(e_1,Jf_1)$, up to replacing $f_1$ with $f'_1:=\frac{1}{1+k^2}f_1+\frac{k}{1+k^2}Jf_1$,  $f_1$ can be chosen such that $$\omega(e_1,Jf_1)=\omega(Je_1,f_1)=\omega(f_1,Je_1)=\omega(Jf_1,e_1)=0.$$
	One can check that $e_1,f_1,Je_1,Jf_1$ are linearly independent and span a linear subspace $W_1$. If $\dim V>2q+4$, one can decompose $W=W_1\oplus W_1^\omega$, where $W_1^\omega:=\{w\in W\ |\ \omega(w,v)=0\ \forall v\in W_1\}$ is $J$-invariant because of the $J$-invariance of $W_1$ and \eqref{eq:omega_J}. As before one can choose $e_2,f_2\in W_1^\omega$ such that $\omega(e_2,f_2)=1$ and $\omega(e_2,Jf_2)=0$, and define a second 4-dimensional vector space $W_2$ spanned by $e_2,f_2,Je_2,Jf_2$. Iterating the argument one gets the result.
\end{proof}
\begin{proposition}\label{Prop.rank}
	Let $(M,\f,\xi,\eta)$ be an anti-normal almost contact manifold of constant rank. Then $\operatorname{rk}(\eta)=4p+1$ and $\dim M=2q+4p+1$, where $2q$ is the rank of the subbundle $\mathcal{E}$ of $TM$ defined by $\mathcal{E}_x:=\{X\in\D_x\ |\ d\eta_x(X,\cdot)=0\}$ for every $x\in M$.
\end{proposition}
\begin{proof}
	Consider the splitting $TM=\D\oplus\langle\xi\rangle$ of the tangent bundle of $M$. In view of (iv) of Proposition \ref{prop-deta-normal}, at every $x\in M$, $d\eta_x$ is a 2-form satisfying \eqref{eq:omega_J} on the complex vector space $(\D_x,J_x:=\f|_{\D_x})$. Thus, by the previous lemma $\dim\D_x=2q+4p$ and $d\eta_x|_{\D_x}$ has rank $4p$, i.e. $(d\eta_x)^{2p}\neq0$ and $(d\eta_x)^{2p+1}=0$ on $\D_x$. Moreover, since $d\eta(\xi,\cdot)=0$, one has that $\eta\wedge(d\eta)^{2p}\neq0$ and $(d\eta)^{2p+1}=0$, namely $\operatorname{rk}(\eta)=4p+1$.
\end{proof}
\smallskip

\begin{remark}
	Unless $d\eta=0$, i.e. $\operatorname{rk}(\eta)=1$, an anti-normal almost contact manifold has dimension at least $5$.
\end{remark}
\medskip

We will see now how the rank of the 1-form $\eta$ measures the non integrability of the almost complex structure $J$ defined in \eqref{eq:J_MxR}.\\

Let $(M,\f,\xi,\eta)$ be an almost contact manifold. Consider the splitting $TM=\D\oplus\langle\xi\rangle$ and the endomorphism $J_\D=\f|_\D$ satisfying $J_\D^2=-I$. By complexification, one has
$$TM^\C=\D^\C\oplus\C\xi=\D^{1,0}\oplus\D^{0,1}\oplus\C\xi,$$
where $\D^{1,0}$ and $\D^{0,1}$ are the eigendistributions associated to eigenvalues $i$ and $-i$ of $J_\D^\C$. One can easily verify that $(\f,\xi,\eta)$ is a normal structure if and only if
$$[\xi,\D^{1,0}]\subset\D^{1,0},\quad [\D^{1,0},\D^{1,0}]\subset \D^{1,0}.$$
In the case of anti-normal structures we have the following:
\smallskip

\begin{proposition}
	An almost contact manifold $(M,\f,\xi,\eta)$ is anti-normal if and only if
	\begin{equation}\label{eq:brackets}
		[\xi,\D^{1,0}]\subset \D^{1,0},\quad [\D^{1,0},\D^{1,0}]_{\D^\C}\subset \D^{1,0},\quad [\D^{1,0},\D^{0,1}]\subset\D^\C.
	\end{equation}
	where $[\cdot,\cdot]_{\D^\C}$ denotes the component along $\D^\C$.
\end{proposition}
\begin{proof}
	We show that \eqref{eq:brackets} is equivalent to \eqref{eq:deta-normal}. The first equation in \eqref{eq:brackets} is equivalent to $N_\f(\xi,\cdot)=0$. Indeed, for any $Z=X-i\f X\in\Gamma(\D^{1,0})$, with $X\in\Gamma(\D)$, we have
	$$[\xi,Z]\in\Gamma(\D^{1,0})\ \Leftrightarrow\ [\xi,\f X]= \f[\xi,X]\ \Leftrightarrow\ (\Lie_\xi\f)X=0,$$
	namely, $N_\f(\xi,X)=0$ by Remark \ref{rmk:N(xi,.)-Lie_f}.
	Now we claim that the second and the third equations in \eqref{eq:brackets} are equivalent, respectively, to $\f(N_\f(X,Y))=0$ and $d\eta(\f X,\f Y)=-d\eta(X,Y)$ for every $X,Y\in\Gamma(\D)$, which in turn are equivalent to $N_\f(X,Y)=2d\eta(X,Y)\xi$ (see Remark \ref{Rmk:aqS_bis}).\\
	Let $Z,W\in\Gamma(\D^{1,0})$, with $Z=X-i\f X$ and $W=Y-i\f Y$, for some $X,Y\in\Gamma(\D)$, and let us denote by $\eta^\C$ the $\C$-linear extension of $\eta$. Then
	\begin{eqnarray*}
		[Z,W]_{\D^\C}&=&[Z,W]-\eta^\C[Z,W]\xi\\
		&=&[X,Y]-[\f X,\f Y]-i([X,\f Y]+[\f X,Y])\\
		&&-\eta([X,Y]-[\f X,\f Y])\xi+i\eta([X,\f Y]+[\f X,Y])\xi.
	\end{eqnarray*}
	Therefore:
	\begin{eqnarray*}
		[Z,W]_{\D^\C}\in\D^{1,0}&\Leftrightarrow&\Im [Z,W]_{\D^\C}=-\f\Re [Z,W]_{\D^\C}\\
		&\Leftrightarrow&-[X,\f Y]-[\f X,Y]+\eta[X,\f Y]\xi+\eta[\f X,Y]\xi=-\f[X,Y]+\f[\f X,\f Y]\\
		&\Leftrightarrow&\f^2[X,\f Y]+\f^2[\f X,Y]=\f[\f X,\f Y]-\f[X,Y]\\
		&\Leftrightarrow&\f(N_\f(X,Y))=0.
	\end{eqnarray*}
Analogously, taking $Z=X-i\f X\in\Gamma(\D^{1,0})$ and $W=Y+i\f Y\in\Gamma(\D^{0,1})$ for some $X,Y\in\Gamma(\D)$ one has:
\begin{eqnarray*}
	[Z,W]\in\D^\C&\Leftrightarrow&\eta^\C[Z,W]=0\\
	&\Leftrightarrow&\eta([X,Y]+[\f X,\f Y])+i\eta([X,\f Y]-[\f X,Y])=0\\
	&\Leftrightarrow&d\eta(X,Y)+d\eta(\f X,\f X)=0.
\end{eqnarray*}
\end{proof}

\begin{remark}
	Differently from the case of normal structures, here the complex distribution $\D^{1,0}$  in general is not involutive: the commutator of any complex fields of type $(1,0)$ $Z=X-i\f X$, $W=Y-i\f Y$, has a component along $\xi$ given by
	\begin{eqnarray*}
		\eta^\C([Z,W])\xi&=&\eta([X,Y]-[\f X,\f Y])\xi-i\eta([X,\f Y]+[\f X,Y])\xi\\
		&=&{}-2d\eta(X,Y)\xi+2id\eta(X,\f Y)\xi,
	\end{eqnarray*}
 where we applied the $\f$-anti-invariance of $d\eta$. Therefore $\D^{1,0}$ is involutive if and only if the $d\eta=0$.\\
 The rank of $\eta$ represents an obstruction even for the integrability of the almost complex structure $J$ on $M\times\R$ defined in \eqref{eq:J_MxR}. Indeed, following \cite[\S 6.1]{Blair}, an easy computation shows that the Nijenhuis tensor is given by
 $$[J,J]\Big(\big(X,a\frac{d}{dt}\big),\big(Y,b\frac{d}{dt}\big)\Big)=\Big(2d\eta(X,Y)\xi,2d\eta(X,\f Y)\frac{d}{dt}\Big).$$

\end{remark}

\subsection{Anti-quasi-Sasakian manifolds}
\begin{definition}\label{Def:aqS}
	We define an \textit{anti-quasi-Sasakian manifold} (aqS manifold for short) as an anti-normal almost contact metric manifold $(M,\f,\xi,\eta,g)$ with closed fundamental 2-form, i.e.
	\begin{equation}\label{eq:aqS}
		d\Phi=0,\quad N_\f=2d\eta\otimes\xi.
	\end{equation}
\end{definition}

Notice that cok\"ahler structures are aqS structures with $\operatorname{rk}(\eta)=1$. In fact the class of cok\"ahler manifolds represents the intersection between quasi-Sasakian and anti-quasi-Sasakian manifolds.
Regarding quasi-Sasakian manifolds, we recall that they are characterized by means of the following identity involving the covariant derivative of the structure tensor $\f$ with respect to the Levi-Civita connection:	
\begin{equation}\label{qS}
	(\nabla_X\f)Y=\eta(Y)AX-g(X,AY)\xi,
\end{equation}	
where $A$ is a $(1,1)$-tensor field such that $g(AX,Y)=g(X,AY)$ and $A\f=\f A$, given by $A=-\f\circ\nabla\xi+k\eta\otimes\xi$, for some smooth function $k$ (see \cite{Kanemaki1,Kanemaki2}).
We prove now a characterization theorem for aqS manifolds.

\begin{theorem}\label{Thm:char.aqS}
An almost contact metric manifold $(M,\f,\xi,\eta,g)$ is anti-quasi-Sasakian if and only if there exists a skew-symmetric $(1,1)$-tensor field $A$ such that $A\f=-\f A$ and
\begin{equation}\label{eq:aqS.char}
	(\nabla_X\f)Y=2\eta(X)AY+\eta(Y)AX+g(X,AY)\xi
\end{equation}
for every $X,Y\in\X(M)$. The tensor field $A$ is uniquely determined by $A=-\f\circ\nabla\xi$.
\end{theorem}
\begin{proof}
Let us assume that $M$ is an aqS manifold. For every $X\in\X(M)$, let us define
$$AX:=(\nabla_X\f)\xi=-\f\nabla_X\xi.$$
Applying \eqref{eq:g(nablaf,.)}, \eqref{eq:aqS} and (iv) of Proposition \ref{prop-deta-normal}, for every $X,Y,Z\in\X(M)$ we get:
\begin{equation}\label{eq:g(nabla,Z)}
	2g((\nabla_X\f)Y,Z)=2d\eta(Y,\f Z)\eta(X)+d\eta(\f Y,X)\eta(Z)-d\eta(\f Z,X)\eta(Y).
\end{equation}
Replacing $Y$ by $\xi$ we obtain:
\begin{equation}\label{eq:g(AX,Z)}
	2g(AX,Z)=d\eta(X,\f Z),
\end{equation}
which implies that $A$ is skew-symmetric with respect to $g$ since $d\eta(X,\f Z)=d\eta(\f X,Z)$.
Using \eqref{eq:g(AX,Z)} in \eqref{eq:g(nabla,Z)} we have:
$$g((\nabla_X\f)Y,Z)=2g(AY,Z)\eta(X)-g(AX,Y)\eta(Z)+g(AX,Z)\eta(Y)$$
which gives \eqref{eq:aqS.char}.
It remains to show that $A$ and $\f$ anticommute each other. First we note that $A\f$ is skew-symmetric with respect to $g$. Indeed from \eqref{eq:g(AX,Z)} it follows that
$$2g(A\f X, Y)=d\eta(\f X,\f Y)=-d\eta(\f Y,\f X)=-2g(A\f Y,X).$$
Then, by the skew-symmetry of $A$, $\f$ and $A\f$, for every $X,Y\in\X(M)$ we get:
$$g(A\f X,Y)=-g(\f X,AY)=g(X,\f AY)=-g(\f AX,Y),$$
which implies that $A\f=-\f A$.

Conversely, assume that there exists a skew-symmetric $(1,1)$-tensor field $A$ which anticommutes with $\f$ and satisfies \eqref{eq:aqS.char}.
Firstly we point out that $A\xi=0$ and $\eta\circ A=0$. Indeed, by the skew-symmetry of $A$, $g(A\xi,\xi)=0$ and for every $X\in\X(M)$
$$g(A\xi,\f X)=-g(\xi,A\f X)=g(\xi,\f AX)=0.$$
Thus $A\xi=0$ and then $\eta(AX)=g(AX,\xi)=-g(X,A\xi)=0$.
Now, applying \eqref{eq:aqS.char} for $Y=\xi$ we have $(\nabla_X\f)\xi=-\f\nabla_X\xi=AX$, that is $\nabla_X\xi=\f AX$. \\
In order to show that the structure is anti-normal, it is convenient to express $N_\f$ as:
\begin{eqnarray}\label{eq:N_f}
	N_\f(X,Y)&=&(\nabla_{\f X}\f)Y-(\nabla_{\f Y}\f)X+(\nabla_X\f)\f Y-(\nabla_Y\f)\f X\nonumber\\
	&&{}+\eta(X)\nabla_Y\xi-\eta(Y)\nabla_X\xi
\end{eqnarray}
for every $X,Y\in\X(M)$. Therefore,
\begin{eqnarray*}
	N_\f(X,Y)&=&\eta(Y)A\f X+g(\f X,AY)\xi-\eta(X)A\f Y-g(\f Y,AX)\xi\\
	&&{}+2\eta(X)A\f Y+g(X,A\f Y)\xi-2\eta(Y)A\f X-g(Y,A\f X)\xi\\
	&&{}+\eta(X)\f AY-\eta(Y)\f AX\\
	&=&4g(\f X,AY)\xi.
\end{eqnarray*}
Moreover,
$$d\eta(X,Y)=g(\nabla_X\xi,Y)-g(X,\nabla_Y\xi)=g(\f AX,Y)-g(X,\f AY)=2g(\f X,AY),$$
where the last equality follows again from the fact that $A$ anticommutes with $\f$ and it is skew-symmetric. Thus we obtained that $N_\f(X,Y)=2d\eta(X,Y)\xi$.
Finally, $\Phi$ is closed. Indeed, for every $X,Y,Z\in\X(M)$
\begin{eqnarray*}
	d\Phi(X,Y,Z)&=&\underset{X,Y,Z}{\mathfrak{S}}(\nabla_X\Phi)(Y,Z)=-\underset{X,Y,Z}{\mathfrak{S}}g((\nabla_X\f)Y,Z)\\
	&=&{}-2\eta(X)g(AY,Z)-\eta(Y)g(AX,Z)-g(X,AY)\eta(Z)\\
	&&{}-2\eta(Y)g(AZ,X)-\eta(Z)g(AY,X)-g(Y,AZ)\eta(X)\\
	&&{}-2\eta(Z)g(AX,Y)-\eta(X)g(AZ,Y)-g(Z,AX)\eta(Y),
\end{eqnarray*}	
which vanishes by the skew-symmetry of $A$.
\end{proof}
\smallskip

\begin{proposition}\label{Prop:aqS-Killing}
	Let $(M,\f,\xi,\eta,g)$ be an anti-quasi-Sasakian manifold. Then the following properties hold:
	\begin{enumerate}
		\item[(i)] $\nabla_\xi\xi=0$;
		\item[(ii)] $\nabla_{\f X}\xi=-\f\nabla_X\xi$;
		\item[(iii)] $\xi$ is Killing.
	\end{enumerate}
\end{proposition}
\begin{proof}
	From $A=-\f\circ \nabla\xi$ we have that $\nabla\xi=\f A=-A\f$, which immediately gives both (i) and (ii). Concerning (iii), since $A$ and $\f$ anticommute and they are skew-symmetric with respect to $g$, so is $\f A=\nabla\xi$.
\end{proof}

\section{AqS manifolds as transversely K\"ahler manifolds}
In this section, first we place the class of anti-quasi-Sasakian manifolds in the framework of the Chinea-Gonzalez classification of almost contact metric manifolds, enlightening the relationship with quasi-Sasakian manifolds. Then we will focus on the transverse geometry of aqS manifolds with respect to the 1-dimensional foliation generated by $\xi$.

\subsection{Chinea-Gonzalez classification}
We breafly recall the Chinea-Gonzalez classifying criterion \cite{Chinea-Gonzalez}.
Given a $(2n+1)$-dimensional real vector space $V$ endowed with an almost contact metric structure $(\f,\xi,\eta,\left<,\right>)$, let $\calC(V)$ be the finite dimensional vector space consisting of all tensors of type $(0,3)$ having the same symmetries of the covariant derivative $\nabla\Phi$ in the case of manifolds, that is
\[\alpha(X,Y,Z)=-\alpha(X,Z,Y)=-\alpha(X,\f Y,\f Z)+\eta(Y)\alpha(X,\xi,Z)+\eta(Z)\alpha(X,Y,\xi)\]
for every $X,Y,Z\in V$. The space $\calC(V)$ decomposes into twelve orthogonal irreducible factors $\calC_i$ ($i=1,\dots,12$) under the action of the group $U(n)\times 1$, providing $2^{12}$ invariant subspaces. In particular the null subspace $\{0\}$ corresponds to the class of cok\"ahler manifolds ($\nabla\Phi=0$).
Next, we will be interested in the following four classes:

\begin{table}[h]
	\centering
	\renewcommand{\arraystretch}{1.3}
	\begin{tabular}{c|l}
		\textbf{Class}&\textbf{Defining condition}\\
		\hline
		$\calC_6$&$\alpha(X,Y,Z)=\frac{1}{2n}[\left<X,Y\right>\eta(Z)-\left<X,Z\right>\eta(Y)](c_{12}\alpha)\xi$\\
		\hline
		$\calC_7$&$\alpha(X,Y,Z)=\eta(Z)\alpha(Y,X,\xi)-\eta(Y)\alpha(\f X,\f Z,\xi)$, $(c_{12}\alpha)\xi=0$\\
		\hline
		$\calC_{10}$&$\alpha(X,Y,Z)=-\eta(Z)\alpha(Y,X,\xi)+\eta(Y)\alpha(\f X,\f Z,\xi)$\\
		\hline
		$\calC_{11}$&$\alpha(X,Y,Z)=-\eta(X)\alpha(\xi,\f Y,\f Z)$\\
		\hline
	\end{tabular}
\caption{\label{Tab.1}}
\end{table}

\noindent where $(c_{12}\alpha)\xi=\sum_i\alpha(e_i,e_i,\xi)$, for any orthonormal basis $\{e_i\}$ of $V$.
Straightforward computations show that
\smallskip

\begin{table}[h!]
	\centering
	\renewcommand{\arraystretch}{1.3}
	\begin{tabular}{c|l}
		\textbf{Class}&\textbf{Defining condition}\\
		\hline
		$\calC_6\oplus \calC_7$& $\alpha(X,Y,Z)=\eta(Z)\alpha(Y,X,\xi)-\eta(Y)\alpha(\f X,\f Z,\xi)$\\
		\hline
		$\calC_{10}\oplus \calC_{11}$& $\alpha(X,Y,Z)=-\eta(Z)\alpha(Y,X,\xi)+\eta(Y)\alpha(\f X,\f Z,\xi)-\eta(X)\alpha(\xi,\f Y,\f Z)$\\
		\hline
		$\calC_6\oplus \calC_7\oplus \calC_{10}\oplus \calC_{11}$&	$\alpha(X,Y,Z)=\eta(Z)\alpha(X,Y,\xi)-\eta(Y)\alpha(\f Z,\f X,\xi)-\eta(X)\alpha(\xi,\f Y,\f Z)$\\
		\hline
	\end{tabular}
	\caption{\label{Tab.2}}
\end{table}

Recall that an almost contact metric structure $(\f,\xi,\eta,g)$ on a smooth manifold $M$ is of class $\calC_6\oplus \calC_7$ if and only if it is quasi-Sasakian. In the following we will provide a characterization of the class $\calC_6\oplus \calC_7\oplus \calC_{10}\oplus \calC_{11}$ which, as we will see later, also includes anti-quasi-Sasakian structures.
From Table \ref{Tab.2} the defining condition of $\calC_6\oplus \calC_7\oplus \calC_{10}\oplus \calC_{11}$ in terms of $\nabla\Phi$ is
\begin{equation*}
	(\nabla_X\Phi)(Y,Z)=\eta(Z)(\nabla_X\Phi)(Y,\xi)-\eta(Y)(\nabla_{\f Z}\Phi)(\f X,\xi)-\eta(X)(\nabla_\xi\Phi)(\f Y,\f Z),
\end{equation*}
which can be written also as
\begin{equation}\label{eq:C6+...+C11}
	(\nabla_X\Phi)(Y,Z)=\eta(Z)(\nabla_X\eta)\f Y+\eta(Y)(\nabla_{\f Z}\eta)X-\eta(X)(\nabla_\xi\Phi)(\f Y,\f Z),
\end{equation}
since $(\nabla_X\Phi)(Y,\xi)=-g((\nabla_X\f)Y,\xi)=-\eta(\nabla_X\f Y)=(\nabla_X\eta)\f Y$ and $(\nabla_X\eta)\xi=0$.\\

The following notion was introduced in \cite{Puhle:GqS} for 5-dimensional almost contact metric manifolds. We extend it for manifolds of any dimension:

\begin{definition}
	An almost contact metric manifold $(M,\f,\xi,\eta,g)$ is called \textit{generalized-quasi-Sasakian} provided $\xi$ is a Killing vector field and $$d\Phi(X,Y,Z)=N_\f(X,Y,Z)=0\quad \forall X,Y,Z\in\Gamma(\D),$$
	where $N_\f(X,Y,Z):=g(N_\f(X,Y),Z)$.
\end{definition}
\smallskip

\begin{proposition}\label{Prop:C6+...+C11}
	An almost contact metric structure $(\f,\xi,\eta,g)$ on a smooth manifold $M$ is of class $\calC_6\oplus\calC_7\oplus\calC_{10}\oplus \calC_{11}$ if and only if it is generalized-quasi-Sasakian.
\end{proposition}
\begin{proof}
	Assume that $(M,\f,\xi,\eta,g)$ belongs to  $\calC_6\oplus\calC_7\oplus\calC_{10}\oplus\calC_{11}$. From \eqref{eq:C6+...+C11} it follows that $(\nabla_X\Phi)(Y,Z)=0$ for every $X,Y,Z\in\Gamma(\D)$, and thus	$$d\Phi(X,Y,Z)=\mathfrak{S}_{X,Y,Z}(\nabla_X\Phi)(Y,Z)=0.$$
	By an equivalent formulation of \eqref{eq:N_f}, we also have
	\begin{eqnarray*}
		N_\f(X,Y,Z)&=&(\nabla_Y\Phi)(X,\f Z)+(\nabla_{\f Y}\Phi)(X,Z)-(\nabla_X\Phi)(Y,\f Z)\\
		&&{}-(\nabla_{\f X}\Phi)(Y,Z)+\eta(Z)d\eta(X,Y)=0.
	\end{eqnarray*}
	Applying \eqref{eq:C6+...+C11} for $X=Z=\xi$ we get that $(\nabla_\xi\eta)\f Y=0$, and hence $\nabla_\xi\eta=0$.
	Moreover,
		$$(\nabla_X\Phi)(\xi,Y)=(\nabla_{\f Y}\eta)X,\qquad (\nabla_X\Phi)(Y,\xi)=(\nabla_X\eta)\f Y,$$
	and thus, by the skew-symmetry of $\nabla_X\Phi$, one has that $(\nabla_{\f Y}\eta)X+(\nabla_X\eta)\f Y=0$. On the other hand $(\nabla_X\eta)\xi=0=(\nabla_\xi\eta)X$ and then, for every $X,Y\in\X(M)$ we get $(\nabla_X\eta)Y+(\nabla_Y\eta)X=0$,
which proves that $\xi$ is a Killing vector field.

Conversely, assume that $(M,\f,\xi,\eta,g)$ is a generalized-quasi-Sasakian manifold. Since $\xi$ is Killing, for every $X,Y\in\X(M)$
\begin{equation}\label{eq:deta=2nabla_eta}
	d\eta(X,Y)=(\nabla_X\eta)Y-(\nabla_Y\eta)X=2(\nabla_X\eta)Y=-2(\nabla_Y\eta)X.
\end{equation}
In particular, $d\eta(\xi,X)=-2(\nabla_X\eta)\xi=0$. Now we show that
\begin{equation}\label{eq:N_dPhi}
	N_\f(\xi,X,\f Y)=d\Phi(\xi,X,Y)
\end{equation}
for every $X,Y\in\X(M)$. Indeed:
\begin{eqnarray*}
	N_\f(\xi,X,\f Y)&=&g(N_\f(\xi,X),\f Y)\\
	&=&g(\f^2[\xi,X]-\f[\xi,\f X]+d\eta(\xi,X)\xi,\f Y)\\
	&=&g(\f[\xi,X],Y)-g([\xi,\f X],Y)+\eta[\xi,\f X]\eta(Y)\\
	&=&g(\f\nabla_\xi X,Y)-g(\f\nabla_X\xi,Y)-g(\nabla_\xi\f X,Y)+g(\nabla_{\f X}\xi,Y)\\
	&=&-g((\nabla_\xi\f)X,Y)+g((\nabla_X\f)\xi,Y)-g(\f X,\nabla_Y\xi)\\
	&=&(\nabla_\xi\Phi)(X,Y)+(\nabla_X\Phi)(Y,\xi)-g(X,(\nabla_Y\f)\xi)\\
	&=&\mathfrak{S}_{\xi,X,Y}(\nabla_\xi\Phi)(X,Y)=d\Phi(\xi,X,Y).
\end{eqnarray*}
Now, by equation \eqref{eq:g(nablaf,.)}, we have
\begin{eqnarray*}
	2(\nabla_\xi\Phi)(\f Y,\f Z)&=&-g((\nabla_\xi\f)\f Y,\f Z)\\
	&=&d\Phi(\xi,\f Y,\f Z)-d\Phi(\xi,\f^2Y,\f^2 Z)+d\eta(Y,\f Z)-d\eta(Z,\f Y).
\end{eqnarray*}
Applying again \eqref{eq:g(nablaf,.)}, and using the fact that $d\Phi(X,Y,Z)=N_\f(X,Y,Z)=0$ along $\D$, together with equations \eqref{eq:N_dPhi} and \eqref{eq:deta=2nabla_eta}, for every $X,Y,Z\in\X(M)$ we get:
\begin{eqnarray*}
	2(\nabla_X\Phi)(Y,Z)
	&=&d\Phi(X,\f^2 Y,\f^2 Z)-\eta(Y)d\Phi(X,\xi,\f^2Z)- \eta(Z)d\Phi(X,\f^2Y,\xi)-d\Phi(X,\f Y,\f Z)\\
	&&{}+\eta(Y)N_\f(\xi,\f^2Z,\f X)+\eta(Z)N_\f(\f^2Y,\xi,\f X)\\
	&&{}-\eta(X)d\eta(\f Y,Z)+\eta(X)d\eta(\f Z,Y)-\eta(Z)d\eta(\f Y,X)+\eta(Y)d\eta(\f Z,X)\\
	&=&\eta(X)[d\Phi(\xi,\f^2Y,\f^2Z)-d\Phi(\xi,\f Y,\f Z)
	-d\eta(\f Y, Z)+d\eta(\f Z,Y)]\\
	&&{}-\eta(Z)d\eta(\f Y,X)+\eta(Y)d\eta(\f Z,X)\\
	&=&{}-2\eta(X)(\nabla_\xi\Phi)(\f Y,\f Z)+2\eta(Z)(\nabla_X\eta)\f Y-2\eta(Y)(\nabla_{\f Z}\eta)X.
\end{eqnarray*}
This proves \eqref{eq:C6+...+C11}, so that the almost contact metric manifold is of class $\calC_6\oplus\calC_7\oplus\calC_{10}\oplus\calC_{11}$.
\end{proof}

\begin{remark}\label{Rmk:N(xi,.,xi)=0}
	Observe that in a generalized-quasi-Sasakian manifold the fact that $\xi$ is Killing gives $d\eta(\xi,\cdot)=0$, or equivalently $N_\f(\xi,\cdot,\xi)=0$.
\end{remark}
\smallskip

The following two results clarify the place of anti-quasi-Sasakian manifolds in the Chinea-Gonzalez classification.

\begin{corollary}\label{Cor:C10+C11}
	A generalized-quasi-Sasakian structure $(\f,\xi,\eta,g)$ on $M$ belongs to $\calC_{10}\oplus\calC_{11}$ if and only if  $$N_\f(X,Y,\xi)=2d\eta(X,Y)\quad \forall X,Y\in\X(M).$$
	In particular every anti-quasi-Sasakian structure is of class $\calC_{10}\oplus\calC_{11}$.
\end{corollary}
\begin{proof}
	From Table \ref{Tab.2}, an almost contact metric structure $(\f,\xi,\eta,g)$ on $M$ belongs to the class $\calC_{10}\oplus\calC_{11}$ if and only if $\nabla\Phi$ satisfies the following condition:
	$$(\nabla_X\Phi)(Y,Z)=-\eta(Z)(\nabla_Y\Phi)(X,\xi)+\eta(Y)(\nabla_{\f X}\Phi)(\f Z,\xi)-\eta(X)(\nabla_\xi\Phi)(\f Y,\f Z)$$
	or equivalently,
	\begin{equation*}
		(\nabla_X\Phi)(Y,Z)=-\eta(Z)(\nabla_Y\eta)\f X-\eta(Y)(\nabla_{\f X}\eta)Z-\eta(X)(\nabla_\xi\Phi)(\f Y,\f Z).
	\end{equation*}
Comparing this with \eqref{eq:C6+...+C11}, we have that a generalized-quasi-Sasakian manifold is of class $\calC_{10}\oplus\calC_{11}$ if and only if $(\nabla_X\eta)\f Y=-(\nabla_Y\eta)\f X$ for every $X,Y\in\X(M)$. Being $\xi$ Killing, this is equivalent to $d\eta(X,\f Y)=-d\eta(Y,\f X)$, namely to the $\f$-anti-invariance of $d\eta$ or $N_\f(X,Y,\xi)=2d\eta(X,Y)$.
\end{proof}
\smallskip

\begin{proposition}
	Let $(M,\f,\xi,\eta,g)$ be an almost contact metric manifold of class $\calC_{10}\oplus\calC_{11}$. Then it is anti-quasi-Sasakian if and only if $\Lie_\xi\f=0$.
\end{proposition}
\begin{proof}
	Since the structure is of class $\calC_{10}\oplus\calC_{11}$, then for every $X,Y,Z\in\Gamma(\D)$
 	$$d\Phi(X,Y,Z)=0,\quad N_\f(X,Y)=2d\eta(X,Y)\xi, \quad N_\f(\xi,X,\xi)=0.$$
 	Hence, taking into account \eqref{eq:deta-normal}, $(\f,\xi,\eta,g)$ is anti-quasi-Sasakian if and only if
	$$d\Phi(\xi,X,Y)=N_\f(\xi,X,\f Y)=0$$ for every $X,Y\in\X(M)$. The first identity is equation \eqref{eq:N_dPhi} already showed in the proof of Proposition \ref{Prop:C6+...+C11} for every generalized-quasi-Sasakian manifold. Owing to $N_\f(\xi,\cdot,\xi)=0$, the vanishing of $N_\f(\xi,X,\f Y)$, is equivalent to $N_\f(\xi,\cdot)=0$, namely $\Lie_\xi\f=0$ by Remark \ref{rmk:N(xi,.)-Lie_f}.
\end{proof}
\medskip
\begin{remark}
	The only anti-quasi-Sasakian manifolds of pure type $\calC_{10}$ or $\calC_{11}$ are cok\"ahler.
	Indeed, from the defining conditions of $\calC_{10}$, $\calC_{11}$ and $\calC_{10}\oplus\calC_{11}$, one has that $M$ is of class $\calC_{10}$ if and only if $$0=-\eta(X)(\nabla_\xi\Phi)(\f Y,\f Z)=-2\eta(X)g(AY,Z)\quad \forall X,Y,Z\in\X(M)$$
	which is equivalent to $A=0$. Similarly, $M$ belongs to $\calC_{11}$ if and only if $$0=-\eta(Z)(\nabla_Y\eta)\f X-\eta(Y)(\nabla_{\f X}\eta)Z=\eta(Z)g(AX,Y)-\eta(Y)g(AX,Z)$$
	for every $X,Y,Z\in\X(M)$. This is equivalent to $g(AX,Y)\xi-\eta(Y)AX=0$ for any $X,Y\in\X(M)$. Being $AX$ orthogonal to $\xi$, it implies $A=0$.
\end{remark}

\subsection{The transverse geometry of aqS manifolds}
Before treating the metric case, we focus on the transverse geometry with respect to the Reeb vector field, determined by the anti-normal condition on an almost contact manifold.

\begin{proposition}
	Every anti-normal almost contact manifold $(M,\f,\xi,\eta)$ locally fibers onto a complex manifold endowed with a closed $2$-form $\omega$ of type $(2,0)$.
	In particular, if $\eta$ is a contact form, then $M$ locally fibers onto a complex symplectic manifold.
\end{proposition}
\begin{proof}
	Consider the local submersion $\pi:M\to M/\xi$. By (iii) of Proposition \ref{prop-deta-normal}, $\f$ descends to an almost complex structure $J$ on $M/\xi$ which turns out to be integrable, being $N_\f(X,Y)_\D=0$ for every $X,Y\in\Gamma(\D)$ . Furthermore, since $\Lie_\xi d\eta=0$ and $d\eta$ is $\f$-anti-invariant, $d\eta$ projects onto  a closed $J$-anti-invariant 2-form $\omega$, i.e. a 2-form of type $(2,0)$.\\
	In particular, if $\eta$ is a contact form, $\omega$ is a symplectic 2-form of type $(2,0)$, thus determining a complex symplectic structure $(J,\omega)$ on $M/\xi$ (see \cite{Besse,Bazzoni}).
\end{proof}

As a consequence we have the following:

\begin{theorem}\label{Thm:local submersion}
	Every anti-quasi-Sasakian manifold $(M,\f,\xi,\eta,g)$ admits a local Riemannian submersion over a K\"ahler manifold endowed with a closed $2$-form $\omega$ of type $(2,0)$.
\end{theorem}
\begin{proof}
	By Proposition \ref{Prop:aqS-Killing} we already know that $\xi$ is a Killing vector field. Therefore, locally, the space of leaves $M/\xi$ is endowed with a Hermitian structure $(J,k)$, with respect to which the projection $\pi:M\to M/\xi$ is actually a local Riemannian submersion. Moreover, denoting by $\Omega$ the fundamental 2-form of $(J,k)$, $d\Phi=0$ immediately gives $d\Omega=0$, which proves that $(J,k)$ is a K\"ahler structure.
\end{proof}
\medskip

Notice that every quasi-Sasakian manifold $(M,\f,\xi,\eta,g)$ locally fibers over a K\"ahler manifold $(M/\xi,J,k)$ and the $\f$-invariant 2-form $d\eta$ projects onto a closed 2-form of type $(1,1)$, with the same rank of $d\eta$. In fact both quasi-Sasakian and anti-quasi-Sasakian manifolds can be viewed as subclasses of transversely K\"ahler almost contact metric manifolds.

Precisely, we say that an almost contact metric manifold $(M,\f,\xi,\eta,g)$ is \textit{transversely K\"ahler} if the structure tensor fields $(\f,g)$ are projectable along the 1-dimensional foliation generated by $\xi$ and induce a transverse K\"ahler structure \cite{CM.dN.Y.}.

\begin{proposition}
	An almost contact metric manifold $(M,\f,\xi,\eta,g)$ is transversely K\"ahler if and only if $$d\Phi=0,\quad N_\f(\xi,X)=0, \quad N_\f(X,Y,Z)=0\quad \forall X,Y,Z\in\Gamma(\D),$$
	i.e. if and only if $M$ is generalized-quasi-Sasakian with $\Lie_\xi\f=0$.
\end{proposition}
\begin{proof}
	According to the above definition $(M,\f,\xi,\eta,g)$ is transversely K\"ahler if and only if
	\begin{equation}\label{eq:transv.Kahler}
		(\Lie_\xi\f)X=0,\ (\Lie_\xi g)(X,Y)=0,\ N_\f(X,Y,Z)=0,\ d\Phi(X,Y,Z)=0
	\end{equation}
	for all $X,Y,Z\in\Gamma(\D)$. By Remark \ref{rmk:N(xi,.)-Lie_f},  the first equation in \eqref{eq:transv.Kahler} is equivalent to $N_\f(\xi,\cdot)=0$.  By a direct computation the following formula holds:
	\begin{equation}\label{eq:dPhi(xi,X,Y)}
		d\Phi(\xi,X,Y)=(\Lie_\xi g)(X,\f Y)+g(X,(\Lie_\xi\f)Y)
	\end{equation}
	for every $X,Y\in\X(M)$. Therefore, when $\Lie_\xi\f=0$, the second equation in \eqref{eq:transv.Kahler} is equivalent to $d\Phi(\xi,X,Y)=0$.
	Regarding the last claim, if $M$ transversely K\"ahler, by \eqref{eq:transv.Kahler}, it is generalized quasi-Sasakian provided that $(\Lie_\xi g)(\xi,\cdot)=0$, but this is consequence of \eqref{eq:dPhi(xi,X,Y)} for $X=\xi$.
\end{proof}

Taking into account the above characterization of transversely K\"ahler almost contact metric manifolds, it is clear that whithin this class, quasi-Sasakian and anti-quasi-Sasakian manifolds are characterized respectively by the $\f$-invariance and $\f$-anti-invariance of $d\eta$, which justifies the name for anti-quasi-Sasakian structures.\\
\smallskip

\begin{figure}[h!]
\centering
\begin{tikzpicture}
	[>=stealth,
	block/.style={rectangle,draw,minimum height=1cm, minimum width=2.8cm}]
		
	\node(1)[block,align=center]{Generalized-\\quasi-Sasakian};
	\node(3)[right=1cm of 1,align=center] {\footnotesize $\Lie_\xi g=0$\\ \footnotesize $d\Phi(X,Y,Z)=0\quad N_\f(X,Y,Z)=0$
	\\ \footnotesize $\forall X,Y,Z\in\Gamma(\D)$};
	\node(7) [left=0.3cm of 1, align=center]{\small $\calC_6\oplus\calC_7\oplus\calC_{10}\oplus\calC_{11}$};
	\node(2)[block,align=center,below=2cm of 1]{Transversely\\K\"ahler}
	[edge from parent/.style={draw,->}, level distance=4cm, sibling distance=5cm,]
		child{node(5)[block,align=center]{quasi-Sasakian} edge from parent node [anchor=east] {\small \scriptsize\parbox{3.5cm}{ \centering $d\eta(\f X,\f Y)=d\eta(X,Y)$\\\medskip$\big(N_\f(X,Y,\xi)=0\big)$}}}
		child{node(6)[block,align=center]{anti-quasi-Sasakian} edge from parent node [anchor=west] {\small\scriptsize\parbox{3.8cm}{  \centering $d\eta(\f X,\f Y)=-d\eta(X,Y)$\\ \medskip $\quad\big(N_\f(X,Y,\xi)=2d\eta(X,Y)\big)$}}};
	\node(4)[right=0.5cm of 2,align=center] {\footnotesize $d\Phi=0\quad N_\f(\xi,X)=0\quad N_\f(X,Y,Z)=0$\\ \footnotesize $\forall X,Y,Z\in\Gamma(\D)$};
	\draw[->] (1)--node[fill=white]{\small $\Lie_\xi\f=0$} (2);
	\node(8)[below=.2cm of 5, align=center]{\small $\calC_6\oplus\calC_7$};
	\node(9)[below=.2cm of 6, align=center]{\small $\calC_{10}\oplus\calC_{11}\ \text{with}\ \Lie_\xi\f=0$};
\end{tikzpicture}
\smallskip

\caption{Quasi-Sasakian and anti-quasi-Sasakian structures within $\calC_6\oplus\calC_7\oplus\calC_{10}\oplus\calC_{11}$.}
\label{Fig:1}
\end{figure}
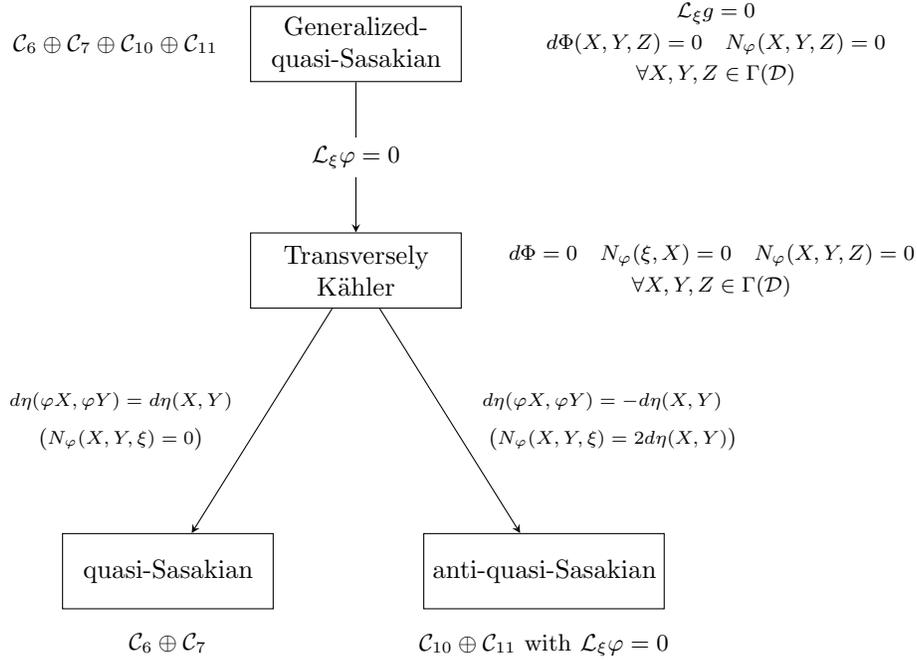
\newpage

As consequence of Theorem \ref{Thm:local submersion}, we have the following Boothby-Wang type theorem.

\begin{theorem}\label{Thm:Boothby-Wang1}
	Let $(M,\f,\xi,\eta,g)$ be an anti-quasi-Sasakian manifold, such that $\xi$ is regular, with compact orbits. Then, up to scaling $\xi$, it generates a free $\mathbb{S}^1$-action on $M$, so that $M$ is a principal circle bundle over a K\"ahler manifold $M/\xi$ endowed with a closed $2$-form $\omega$ of type $(2,0)$.
	In particular, $\eta$ is a connection form on $M$ and its curvature form is given by $d\eta=\pi^*\omega$, where $\pi:M\to M/\xi$ is the bundle projection.
\end{theorem}

Conversely we have the following:

\begin{theorem}\label{Thm:Boothby-Wang2}
	Let $(B,J,k)$ be a K\"ahler manifold endowed with a closed $2$-form $\omega$ of type $(2,0)$ which represents an element of the integral cohomology group $H^2(B,\mathbb{Z})$. Then there exist a principal circle bundle $M$ over $B$ and a connection form $\eta$ on $M$ such that the curvature form is given by $d\eta=\pi^*\omega$.
	Moreover, $M$ is endowed with an anti-quasi-Sasakian structure $(\f,\xi,\eta,g)$.
\end{theorem}
\begin{proof}
	The first part is due to the theorem of Kobayashi (see \cite[Theorem 2.5]{Blair}). Let $\eta$ be a connection form on $M$ such that $d\eta=\pi^*\omega$ and let us consider $\xi\in\X(M)$ as the fundamental vector field associated to the generator $1\equiv\frac{d}{dt}$ of the abelian Lie algebra $\R$ of $\mathbb{S}^1$. Then $\xi$ is vertical and $\eta(\xi)=1$.
	We define the $(1,1)$-tensor field $\f$ by 
	$$\f\xi=0, \qquad \f X^*=(JX)^*,$$
	where $X^*$ denotes the horizontal lift of a vector field $X\in\X(B)$. Then one has that $\eta\circ\f=0$ and
	$$\f^2X^*=(J^2X)^*=-X^*,$$
	so that $\f^2=-I+\eta\otimes\xi$.
	Moreover, let us define the Riemannian metric $g=\pi^*k+\eta\otimes\eta$ on $M$. For every $X,Y\in\X(B)$ it satisfies:
	$$g(\f X^*,\f Y^*)=k(JX,JY)\circ\pi=k(X,Y)\circ\pi=g(X^*,Y^*).$$
	Being also $g(X^*,\xi)=0$ and $g(\xi,\xi)=1$, $(\f,\xi,\eta,g)$ is an almost contact metric structure on $M$. Clearly, its fundamental 2-form $\Phi$ is given by $\Phi=\pi^*\Omega$, being $\Omega$ the fundamental 2-form of the K\"ahler structure $(J,k)$ on $B$. Thus $d\Phi=\pi^*(d\Omega)=0$.	Now, for every $X,Y\in\X(B)$,
	$$[X^*,\xi]=0,\quad h[X^*,Y^*]=[X,Y]^*,\quad \pi_*[X^*,Y^*]=[X,Y],$$
	where $h$ denotes the horizontal component. Then we obtain
\[N_\f(X^*,\xi)=\f^2[X^*,\xi]-\f[\f X^*,\xi]+d\eta(X^*,\xi)
		=-\f[(JX)^*,\xi]-\eta[X^*,\xi]=0,\]
and also
\begin{eqnarray*}
	N_\f(X^*,Y^*)&=&\f^2[X^*,Y^*]+[\f X^*,\f Y^*]-\f[\f X^*,Y^*]-\f[X,\f Y^*]+d\eta(X^*,Y^*)\xi\\
	&=&(J^2\pi_*[X^*,Y^*])^*+[(JX)^*,(JY)^*]-(J\pi_*[(JX)^*,Y^*])^*\\&&{}-(J\pi_*[X^*,(JY)^*])^*+d\eta(X^*,Y^*)\xi\\
	&=&(J^2[X,Y])^*+[JX,JY]^*+\eta([(JX)^*,(JY)^*])\xi-(J[JX,Y])^*\\&&{}-(J[X,JY])^*+d\eta(X^*,Y^*)\xi\\
	&=&(N_J(X,Y))^*-d\eta((JX)^*,(JY)^*)\xi+d\eta(X^*,Y^*)\xi\\
	&=&{}(-\omega(JX,JY)\circ\pi+\omega(X,Y)\circ\pi)\xi\\
	&=&2(\omega(X,Y)\circ\pi)\xi=2d\eta(X^*,Y^*)\xi.
\end{eqnarray*}
This shows that $(M,\f,\xi,\eta,g)$ is an anti-quasi-Sasakian manifold.
\end{proof}

A well known example of K\"ahler manifolds endowed with a closed $(2,0)$-form is given by hyperk\"ahler manifolds (see for instance \cite[14.B]{Besse}).  In this regards we point out that, in the hypotheses of Theorem \ref{Thm:Boothby-Wang1}, if $M$ is compact and the anti-quasi-Sasakian structure has maximal rank, then the base manifold $M/\xi$ is a compact K\"ahler manifold endowed with a complex symplectic structure. In fact, a result due to A. Beauville (\cite{Beauville}, or \cite[Theorem 14.16]{Besse}) ensures that there exists a hyperk\"ahler metric on $M/\xi$.

In the next section we will see how anti-quasi-Sasakian structures naturally appear on a special class of Riemannian manifolds with a 1-dimensional foliation whose space of leaves admits a hyperk\"ahler structure.

\section{Special classes of anti-quasi-Sasakian manifolds}\label{Sec:examples}
\subsection{$Sp(n)$-almost contact metric structures}
\begin{definition}
	We call \textit{$Sp(n)$-almost contact metric manifold} any smooth manifold $M$ admitting three almost contact metric structures $(\f_i,\xi,\eta,g)$, $i=1,2,3$, sharing the same Reeb vector field $\xi$ and Riemannian metric $g$, and satisfying the quaternionic identities
	\begin{equation}\label{eq:quaternionic id.}
		\f_i\f_j=\f_k=-\f_j\f_i
	\end{equation}
	for every even permutation $(i,j,k)$ of $(1,2,3)$.
	This is equivalent to require that the structure group of the frame bundle is reducible to $Sp(n)\times 1$. 	
\end{definition}

An $Sp(n)$-almost contact metric manifold $M$ has dimension $4n+1$. The tangent bundle splits as $TM=\D\oplus\left<\xi\right>$,
where $\D=\operatorname{Ker}\eta=\operatorname{Im}\f_i=\left<\xi\right>^\perp$, $i=1,2,3$.
In particular the manifold admits local orthonormal frames of type $\{X_l,\f_1X_l,\f_2X_l,\f_3X_l,\xi\}$, $l=1,\dots,n$.
In the following an $Sp(n)$-almost contact metric structure will be denoted by $(\f_i,\xi,\eta,g)$, omitting to specify that the index $i$ runs in $\{1,2,3\}$.
\medskip

\begin{remark}
	For a $4n$-dimensional Riemannian manifold $(M,g)$, an $Sp(n)$-reduction of the structure group of the frame bundle is equivalent to the existence of three compatible almost complex structures $J_1,J_2,J_3$ satisfying $J_iJ_j=J_k=-J_jJ_i$ for every even permutation $(i,j,k)$ of $(1,2,3)$. One says that $(J_1,J_2,J_3,g)$ is an almost hyperhermitian structure.
	For these manifolds the Hitchin's Lemma \cite{Hitchin} states that if the three fundamental $2$-forms are closed, then all the almost complex structures are integrable, and thus the manifold is hyperk\"ahler.
\end{remark}
\medskip

For an $Sp(n)$-almost contact metric manifold we will see how the closedness of the fundamental 2-forms $\Phi_i$, $i=1,2,3$, effects on the normality of the almost contact structures. To this aim we will use the following lemma; the proof is omitted since it is exactly the same as in \cite[Lemma 4.1]{CappellettiM.-Dileo}, where it is shown in the 5-dimensional case for an $SU(2)$-structure.

\begin{lemma}\label{Lemma:CM-D}
	Let $(M,\f_i,\xi,\eta,g)$ be an $Sp(n)$-almost contact metric manifold. Then
	\begin{eqnarray*}\label{eq:lemmaCM-D}
		g(N_{\f_i}(X,Y),\f_jZ)&=&d\Phi_j(X,Y,Z)-d\Phi_j(\f_iX,\f_iY,Z)\nonumber\\
		&&-d\Phi_k(\f_iX,Y,Z)-d\Phi_k(X,\f_iY,Z),
	\end{eqnarray*}
	for every $X,Y,Z\in\X(M)$ and $(i,j,k)$ even permutation of $(1,2,3)$.
\end{lemma}
\medskip

\begin{theorem}\label{Thm:Sp(n)-struct.}
	Let $(M,\f_i,\xi,\eta,g)$ be an $Sp(n)$-almost contact metric manifold such that
	\begin{equation}\label{eq:hypothesis}
		d\Phi_1=0,\quad d\Phi_2=0,\quad d\eta=2\Phi_3.
	\end{equation}
	Then $(\f_1,\xi,\eta,g)$ and $(\f_2,\xi,\eta,g)$ are anti-quasi-Sasakian structures, while $(\f_3,\xi,\eta,g)$ is a Sasakian structure.
	In particular, $M$ locally fibers onto a hyperk\"ahler manifold.
\end{theorem}
\begin{proof}
	Since the fundamental 2-forms are closed,
	Lemma \ref{Lemma:CM-D} implies $g(N_{\f_i}(X,Y),\f_jZ)=0$ for every $X,Y,Z\in\X(M)$, that is $N_{\f_i}(X,Y)_\D=0$ for $i=1,2,3$.
	Now observe that the fundamental 2-form $\Phi_3$ is $\f_3$-invariant, while for $i=1,2$, it is $\f_i$-anti-invariant. Indeed,
	$$\Phi_3(\f_iX,\f_iY)=g(\f_iX,\f_3\f_iY)=-g(\f_iX,\f_i\f_3Y)
	=-g(X,\f_3Y)=-\Phi_3(X,Y).$$
	Since $d\eta=2\Phi_3$, as argued in Remark \ref{Rmk:aqS_bis},
	$$\eta(N_{\f_3}(X,Y))=0,\qquad \eta(N_{\f_i}(X,Y))=2d\eta(X,Y).$$ 	
	Therefore, the structure $(\f_3,\xi,\eta,g)$ is normal and hence Sasakian, while $(\f_i,\xi,\eta,g)$, $i=1,2$, are anti-normal and hence anti-quasi-Sasakian.
	
	Now, considering the local submersion $\pi:M\to M/\xi$, by Proposition \ref{prop-deta-normal} and the fact that the third structure is Sasakian, all the structure tensor fields  are projectable onto an almost hyperhermitian structure $(J_1,J_2,J_3,k)$.
	Denoting by $\Omega_i$ the fundamental 2-form of $(J_i,k)$, one has that $\pi^*\Omega_i=\Phi_i$. Therefore $d\Omega_i=0$ for every $i=1,2,3$ and then $(J_i,k)$ is a hyperk\"ahler structure by the Hitchin's Lemma.
\end{proof}
\smallskip

\begin{definition}\label{Def:double_aqS}
	An $Sp(n)$-almost contact metric structure $(\f_i,\xi,\eta,g)$ satisfying \eqref{eq:hypothesis} will be called a \textit{double aqS-Sasakian structure}.
\end{definition}
\smallskip

Notice that for a double aqS-Sasakian manifold the two anti-quasi-Sasakian structures $(\f_1,\xi,\eta,g)$ and $(\f_2,\xi,\eta,g)$ project onto K\"ahler structures for which the $(2,0)$-form $\omega$ claimed in Theorem \ref{Thm:local submersion} coincides with $\Omega_3$.
The following proposition characterizes double aqS-Sasakian manifolds in terms of the Levi-Civita connection.

\begin{proposition}\label{Prop:char.Sp(n)-str.}
	An $Sp(n)$-almost contact metric manifold $(M,\f_i,\xi,\eta,g)$  is double aqS-Sasakian if and only if any two of the following conditions hold:
	\begin{enumerate}
		\item[(i)] $(\nabla_X\f_1)Y=-2\eta(X)\f_2Y-\eta(Y)\f_2X-g(X,\f_2Y)\xi;$
		\item[(ii)] $(\nabla_X\f_2)Y=2\eta(X)\f_1Y+\eta(Y)\f_1X+g(X,\f_1Y)\xi;$
		\item[(iii)] $(\nabla_X\f_3)Y=g(X,Y)\xi-\eta(Y)X$.
	\end{enumerate}
In particular, each one of the above three conditions is consequence of the other two.
\end{proposition}
\begin{proof}
	If $(M,\f_i,\xi,\eta,g)$ is a double aqS-Sasakian manifold, Theorem \ref{Thm:Sp(n)-struct.} holds. In particular, since $(\f_3,\xi,\eta,g)$ is a Sasakian structure, it satisfies (iii) and $\nabla\xi=-\f_3$. Then, (i) and (ii) follow from Theorem \ref{Thm:char.aqS}, where the operators $A_{\f_1}$ and $A_{\f_2}$ associated to the two anti-quasi-Sasakian structures are given by $A_{\f_1}=-\f_1\circ\nabla\xi=-\f_2$ and $A_{\f_2}=-\f_2\circ\nabla\xi=\f_1$.
	
	Conversely, first assume that (i) and (ii) hold. Since $\f_1$ and $\f_2$ are skew-symmetric and anticommute each other, by Theorem \ref{Thm:char.aqS} both the structures $(\f_i,\xi,\eta,g)$, $i=1,2$, are anti-quasi-Sasakian and thus $d\Phi_1=d\Phi_2=0$. Applying (i) or (ii) for $Y=\xi$, one has that $\nabla_X\xi=-\f_3X$, and then
	$$d\eta(X,Y)=g(\nabla_X\xi,Y)-g(X,\nabla_Y\xi)=2g(X,\f_3Y)=2\Phi_3(X,Y),$$
	which proves \eqref{eq:hypothesis}.
	Now assume that (i) and (iii) hold. As above $(\f_1,\xi,\eta,g)$ is anti-quasi-Sasakian and hence $d\Phi_1=0$. On the other hand (iii) means that  $(\f_3,\xi,\eta,g)$ is a Sasakian structure, and hence $d\eta=2\Phi_3$. Moreover, since $\f_2=\f_3\f_1$, one has that $\nabla\f_2=\nabla\f_3\circ\f_1+\f_3\circ\nabla\f_1$ and thus (ii) follows form (i) and (iii). Therefore $d\Phi_2=0$ as before.
	Finally, assuming (ii) and (iii), one gets the conclusion arguing analogously.
\end{proof}
\medskip

We conclude the section treating the $5$-dimensional case, namely $n=1$, for which $Sp(1)=SU(2)$. In \cite{Conti-Salamon}, $SU(2)$-structures on $5$-dimensional manifolds are described in terms of a quadruplet $(\eta,\omega_1,\omega_2,\omega_3)$, where $\eta$ is a 1-form and $\omega_i$, for $i=1,2,3$, are $2$-forms satisfying
	\begin{equation}\label{eq:SU(2)-structure}
		\omega_i\wedge\omega_j=\delta_{ij}v,
	\end{equation}
	for some $4$-form $v$ such that $v\wedge\eta\neq0$, and $\omega_1(X,\cdot)=\omega_2(Y,\cdot)\ \Rightarrow\ \omega_3(X,Y)\ge0.$
	These $2$-forms are related to the underlying almost contact metric structures $(\f_i,\xi,\eta,g)$ by $\omega_i=-\Phi_i$ (see for instace \cite{CappellettiM.-Dileo}). Therefore, equations in \eqref{eq:hypothesis} defining a double aqS-Sasakian structure can be rephrased as
	\begin{equation}\label{eq:SU(2)}
		d\omega_1=0,\quad d\omega_2=0,\quad d\eta=-2\omega_3.
	\end{equation}
	We remark that the class of double aqS-Sasakian manifolds coincides with the class of contact Calabi-Yau $5$-manifolds defined as in \cite{TV}.
	Indeed, assuming \eqref{eq:SU(2)}, Theorem \ref{Thm:Sp(n)-struct.} ensures that $(\f_3,\xi,\eta,g)$ is Sasakian. Furthermore, $\epsilon=\omega_1+i\omega_2$ is a nowhere vanishing basic complex $(2,0)$-form on $\D$, satisfying $d\epsilon=0$ and $\epsilon\wedge\bar\epsilon=\frac12(d\eta)^2$. Conversely, if $(M^5,\f,\xi,\eta,g,\epsilon)$ is contact Calabi-Yau, a double aqS-Sasakian structure is defined by $\omega_1=\Re(\epsilon)$, $\omega_2=\Im(\epsilon)$ and $\omega_3=-\frac12d\eta$.
	
	Finally, as observed in \cite{dAFFU}, these are a special type of $K$-contact hypo $SU(2)$-structures, i.e. $\xi$ is a Killing vector field and
	\begin{equation}\label{eq:contact-hypo}
		d(\eta\wedge\omega_1)=0,\quad d(\eta\wedge\omega_2)=0,\quad d\eta=-2\omega_3.	
	\end{equation}
	Indeed by \eqref{eq:SU(2)} and \eqref{eq:SU(2)-structure}, for $i=1,2$ $$d(\eta\wedge\omega_i)=d\eta\wedge\omega_i=-2\omega_3\wedge\omega_i=2\delta_{3i}v=0.$$
	The converse is not true, although a $K$-contact hypo $5$-manifold still carries a Sasakian structure as stated in the following:
	
\begin{proposition}
	Let $(M,\eta,\omega_1,\omega_2,\omega_3)$ be a $5$-dimensional manifold with a $K$-contact hypo $SU(2)$-structure. Let $(\f_i,\xi,\eta,g)$, $i=1,2,3$, be the underlying  almost contact metric structures. Then $(\f_3,\xi,\eta,g)$ is Sasakian and $(\f_i,\xi,\eta,g)$, $i=1,2$, are generalized-quasi-Sasakian structures of class $\calC_{10}\oplus\calC_{11}$.
	Moreover, $M$ is a double aqS-Sasakian manifold (or equivalently the structure is contact Calabi-Yau) if and only if $\Lie_\xi\f_1$ or $\Lie_\xi\f_2$ vanishes.
\end{proposition}
\begin{proof}
	If \eqref{eq:contact-hypo} holds, taking $i=1,2$ and $X,Y,Z\in\Gamma(\D)$ one has that
	$$0=d(\eta\wedge\omega_i)(\xi,X,Y,Z)=(d\eta\wedge\omega_i)(\xi,X,Y,Z)-(\eta\wedge d\omega_i)(\xi,X,Y,Z)=-d\omega_i(X,Y,Z),$$
	where in the third equality we used  $\omega_i(\xi,X)=0$ and $d\eta(\xi,X)=0$. Therefore $d\Phi_i(X,Y,Z)=-d\omega_i(X,Y,Z)=0$ for every $X,Y,Z\in\Gamma(\D)$. Being also $d\Phi_3=-d\omega_3=0$, Lemma \ref{Lemma:CM-D} implies $N_{\f_i}(X,Y,Z)=0$ for every $X,Y,Z\in\Gamma(\D)$ and for every $i=1,2,3$.
	Now, arguing as in the proof of Theorem \ref{Thm:Sp(n)-struct.}, since $d\eta=2\Phi_3$ is $\f_3$-invariant and $\f_i$-anti-invariant for $i=1,2$, one has $$N_{\f_1}(X,Y,\xi)=N_{\f_2}(X,Y,\xi)=2d\eta(X,Y),\quad N_{\f_3}(X,Y,\xi)=0.$$
	Being $\xi$ Killing by assumption, this proves that the almost contact metric structures $(\f_1,\xi,\eta,g)$ and $(\f_2,\xi,\eta,g)$ belong to $\calC_{10}\oplus\calC_{11}$ (see Corollary \ref{Cor:C10+C11}).
	Moreover, from equation \eqref{eq:dPhi(xi,X,Y)} and $\Lie_\xi g=0$, we have $\Lie_\xi\f_3=0$ and hence $N_{\f_3}(\xi,\cdot)=0$ (see Remark \ref{rmk:N(xi,.)-Lie_f}). Therefore $(\f_3,\xi,\eta,g)$ is normal and thus Sasakian.
	
	Concerning the last statement, $(M,\f_i,\xi,\eta,g)$, $i=1,2,3$, is a double aqS-Sasakian manifold if and only if $d\Phi_1(\xi,\cdot,\cdot)=d\Phi_2(\xi,\cdot,\cdot)=0$. By equation \eqref{eq:dPhi(xi,X,Y)} this is equivalent to $\Lie_\xi\f_1=\Lie_\xi\f_2=0$. Assuming $\Lie_\xi\f_1$ identically zero, $(\f_1,\xi,\eta,g)$ is aqS with $A_{\f_1}=-\f_1\circ\nabla\xi=\f_1\f_3=-\f_2$; thus, applying Proposition \ref{Prop:char.Sp(n)-str.}, also $(\f_2,\xi,\eta,g)$ is aqS. The same holds assuming $\Lie_\xi\f_2=0$.
\end{proof}

\subsection{Weighted Heisenberg Lie groups and compact nilmanifolds} \label{Ex:w-Heisenberg}
	Let $G$ be a $(4n+1)$-dimensional Lie group with Lie algebra $\mathfrak{g}$, and let 
	$\xi,\tau_r,\tau_{n+r},\tau_{2n+r},\tau_{3n+r}$, $r=1,\dots,n$, be a basis of $\mathfrak{g}$. We consider three left invariant almost contact metric structures $(\f_i,\xi,\eta,g)$, $i=1,2,3$, where $g$ is the Riemannian metric with respect to which the basis is orthonormal, $\eta$ is the dual form of $\xi$, and $\f_i$ is given by
	$$\f_i=\sum_{r=1}^{n}\big(\theta_r\otimes\tau_{in+r}-\theta_{in+r}\otimes\tau_r+\theta_{jn+r}\otimes\tau_{kn+r}-\theta_{kn+r}\otimes\tau_{jn+r}\big),$$
	where $\theta_l$ ($l=1,\dots,4n$) is the dual 1-form of $\tau_l$, and $(i,j,k)$ is an even permutation of $(1,2,3)$. Explicitly, $\f_i(\xi)=0$ and
	$$\begin{array}{llll}
		\f_1\tau_r=\tau_{n+r}&\f_1\tau_{n+r}=-\tau_r&\f_1\tau_{2n+r}=\tau_{3n+r}&\f_1\tau_{3n+r}=-\tau_{2n+r}\\
		\f_2\tau_r=\tau_{2n+r}&\f_2\tau_{n+r}=-\tau_{3n+r}&\f_2\tau_{2n+r}=-\tau_r&\f_2\tau_{3n+r}=\tau_{n+r}\\
		\f_3\tau_r=\tau_{3n+r}&\f_3\tau_{n+r}=\tau_{2n+r}&\f_3\tau_{2n+r}=-\tau_{n+r}&\f_3\tau_{3n+r}=-\tau_r.\\		
	\end{array}$$
	Clearly $\f_1,\f_2,\f_3$ satisfy \eqref{eq:quaternionic id.}, so that $(\f_i,\xi,\eta,g)$ is an $Sp(n)$-almost contact metric structure. The fundamental 2-forms are
	$$\Phi_i=-\sum_{r=1}^{n}(\theta_r\wedge\theta_{in+r}+\theta_{jn+r}\wedge\theta_{kn+r}).$$
	Now assume that the nonvanishing commutators are
	$$[\tau_r,\tau_{3n+r}]=[\tau_{n+r},\tau_{2n+r}]=2\lambda_r\xi,$$
	for some constants $\lambda_1,\dots,\lambda_n\in\R$.
	The 1-forms $\theta_l$ are closed for every $l=1,\dots,4n$ and then $d\Phi_i=0$ for every $i=1,2,3$.
	Furthermore
	$$d\eta=-2\sum_{r=1}^n\lambda_r(\theta_r\wedge\theta_{3n+r}+\theta_{n+r}\wedge\theta_{2n+r}).$$
	Notice that when the weights $\lambda_1,\dots,\lambda_n$ are all equal to 1, then $\mathfrak{g}$ is the real Heisenberg Lie algebra of dimension $4n+1$. Furthermore, $d\eta=2\Phi_3$, and thus  $(G,\f_i,\xi,\eta,g)$ is a double aqS-Sasakian manifold. In particular, $(\f_3,\xi,\eta,g)$ is the standard Sasakian structure on the real Heisenberg Lie algebra.
	
    Now we show that in the general case, $(\f_1,\xi,\eta,g)$ and $(\f_2,\xi,\eta,g)$ are anti-quasi-Sasakian structures, while $(\f_3,\xi,\eta,g) $ is quasi-Sasakian.
	From the expression of the Lie brackets, one immediately has that $$g(N_{\f_i}(X,Y),\theta_l)=0,$$
	for every $X,Y\in\mathfrak{g}$ and $l=1,\dots,4n$. In fact this is coherent with Lemma \ref{Lemma:CM-D} since all the fundamental 2-forms are closed. It remains to compute
	$$g(N_{\f_i}(X,Y),\xi)=-d\eta(\f_iX,\f_i Y)+d\eta(X,Y).$$
	For $i=1$:
	\begin{eqnarray*}
		d\eta(\f_1 X,\f_1 Y)
		&=&{}-2\sum_{r=1}^n\lambda_r[\theta_r(\f_1X)\theta_{3n+r}(\f_1Y)-\theta_{3n+r}(\f_1X)\theta_r(\f_1Y)\\
		&&{}+\theta_{n+r}(\f_1X)\theta_{2n+r}(\f_1Y)-\theta_{2n+r}(\f_1X)\theta_{n+r}(\f_1Y)]\\
		&=&2\sum_{r=1}^n\lambda_r[\theta_{n+r}(X)\theta_{2n+r}(Y)-\theta_{2n+r}(X)\theta_{n+r}(Y)\\
		&&{}+\theta_{r}(X)\theta_{3n+r}(Y)-\theta_{3n+r}(X)\theta_{r}(Y)]\\
		&=&{}-d\eta(X,Y).
	\end{eqnarray*}
Thus we have proved that $N_{\f_1}(X,Y)=2d\eta(X,Y)\xi$. Then $(\f_1,\xi,\eta,g)$ is an anti-quasi-Sasakian structure. Analogously one obtains that $N_{\f_2}=2d\eta\otimes\xi$ and $N_{\f_3}=0$, which prove that $(\f_2,\xi,\eta,g)$ is anti-quasi-Sasakian and $(\f_3,\xi,\eta,g)$ is quasi-Sasakian.\\

Now we determine the operators $A_{\f_1}$ and $A_{\f_2}$ associated to the two anti-quasi-Sasakian structures.
From identity \eqref{eq:g(AX,Z)}, $2g(A_{\f_i}X,Y)=d\eta(X,\f_iY)$ ($i=1,2$), and with similiar computation as above we get:
\begin{eqnarray*}
	&&A_{\f_1}=-\sum_{r=1}^n\lambda_r(\theta_r\otimes\tau_{2n+r}-\theta_{2n+r}\otimes\tau_{r}+\theta_{3n+r}\otimes\tau_{n+r}-\theta_{n+r}\otimes\tau_{3n+r});\\
	&&A_{\f_2}=\sum_{r=1}^n\lambda_r(\theta_r\otimes\tau_{n+r}-\theta_{n+r}\otimes\tau_{r}+\theta_{2n+r}\otimes\tau_{3n+r}-\theta_{3n+r}\otimes\tau_{2n+r}).
\end{eqnarray*}
Explicitly $A_{\f_1}\xi=A_{\f_2}\xi=0$ and
$$\begin{array}{llll}
	A_{\f_1}\tau_r=-\lambda_r\tau_{2n+r}& A_{\f_1}\tau_{n+r}=\lambda_r\tau_{3n+r}& A_{\f_1}\tau_{2n+r}=\lambda_r\tau_{r}& A_{\f_1}\tau_{3n+r}=-\lambda_r\tau_{n+r}\\
	A_{\f_2}\tau_r=\lambda_r\tau_{n+r}& A_{\f_2}\tau_{n+r}=-\lambda_r\tau_{r}& A_{\f_2}\tau_{2n+r}=\lambda_r\tau_{3n+r}& A_{\f_2}\tau_{3n+r}=-\lambda_r\tau_{2n+r}.
\end{array}$$
When the weights are equal to 1 we have that $A_{\f_1}=-\f_2$ and $A_{\f_2}=\f_1$ as in Proposition \ref{Prop:char.Sp(n)-str.}. However, in general $A_{\f_1}$ and $A_{\f_2}$ do not define almost contact structures.
\medskip

The Lie group $G$ is a 2-step nilpotent Lie group and the only eventually non zero structure constants of its Lie algebra are the weights $\lambda_r$. A result due to A. I. Malcev \cite{Malcev} ensures that when all the weights $\lambda_r$ are rational numbers, then $G$ admits a cocompact discrete subgroup $\Gamma$, so that an anti-quasi-Sasakian structure on the compact nilmanifold $G/\Gamma$ is induced.

\subsection{Further examples}
\begin{example}\label{ex:compact-Cortes}
	Given a hyperk\"ahler manifold $(B,J_i,k)$ ($i=1,2,3$), the K\"ahler forms $\Omega_2,\Omega_3$ associated to $J_2,J_3$ respectively, are closed and $J_1$-anti-invariant. Thus, if they define integral cohomology classes, they determine principal circle bundles over $B$ endowed with anti-quasi-Sasakian structures in view of Theorem \ref{Thm:Boothby-Wang2}.
	At least locally, one can always consider an open contractible set $U\subset B$ on which the K\"ahler forms are exact, thus defining the trivial cohomology class which corresponds to the trivial bundle $U\times\mathbb{S}^1$. \\
	In \cite{Cortes} V. Cortés shows the existence of non-flat compact hyperk\"ahler manifolds with integral 2-forms, obtained as products of $m$ copies of a $K3$ surface. In particular, applying Theorem \ref{Thm:Boothby-Wang2} to these manifolds, one obtains examples of compact aqS manifolds, which are not quotients of the weighted Heisenberg Lie group, as it is transversely flat.
\end{example}
\smallskip

\begin{example}\label{Ex:disco}
	Let $(B^{4n},J,k)$ be a K\"ahler manifold and fix a coordinate neighborhood $U$ with respect to which the complex structure $J$ is given by $J\frac{\partial}{\partial x_i}=\frac{\partial}{\partial y_i}$ and $J\frac{\partial}{\partial y_i}=-\frac{\partial}{\partial x_i}$, $i=1,\dots,2n$. On $U$ let us consider the following 2-form  $$\omega=\sum_{i=1}^{p}(dx_i\wedge dx_{n+i}-dy_i\wedge dy_{n+i}),$$
	where $1\le p\le n$. Clearly $\omega$ is a 2-form of type $(2,0)$ and rank $4p$. Moreover it is exact, being $\beta=\sum_{i=1}^{p}(x_idx_{n+i}-y_idy_{n+i})$ a primitive 1-form.
	Thus $\omega$ defines the trivial cohomology class on $U$ and hence the trivial bundle $U\times\mathbb{S}^1$ is endowed with an anti-quasi-Sasakian structure $(\f,\xi,\eta,g)$, where the connection form $\eta$ and the Riemannian metric $g$ are given by $$\eta=dt+\pi^*\beta,\quad g=\pi^*k+\eta\otimes\eta.$$
	We point out that, in order to have a global example, one can apply the above construction to the complex unit disc $D^{2n}\subset\C^{2n}$ endowed with the K\"ahler structure of constant holomorphic sectional curvature $c<0$, or to every Hermitian symmetric space of non-compact type, since these can be realized as bounded symmetric domains in $\C^n$ (see \cite[Vol. II, Chap. XI.9]{KN}).
\end{example}
\smallskip

\begin{example}
	Let $(M^{2n+1},\f,\xi,\eta,g)$ be an anti-quasi-Sasakian manifold and $(M^{2m},J,h)$ be a K\"ahler manifold. Then the product manifold $M^{2n+1}\times M^{2m}$ is an anti-quasi Sasakian manifold with respect to the almost contact metric structure $(\tilde\f,\tilde\xi,\tilde\eta,\tilde{g})$ defined by:
	$$\tilde\f X=(\f X_1,JX_2),\quad \tilde\xi=(\xi,0),\quad \tilde\eta(X)=\eta(X_1),\quad \tilde{g}(X,Y)=g(X_1,Y_1)+h(X_2,Y_2),$$
	where $X=(X_1,X_2),\ Y=(Y_1,Y_2)\in\X(M^{2n+1}\times M^{2m})$. Indeed, denoting by $\Omega$ the K\"ahler form of $M^{2m}$, and by $\Phi,\tilde\Phi$ the fundamental 2-forms of the structures $(\f,\xi,\eta,g)$ and $(\tilde\f,\tilde\xi,\tilde\eta,\tilde{g})$ respectively, one has $\tilde\Phi=\Phi+\Omega,$ and hence $d\tilde\Phi=0$. Moreover the $N_{\tilde\f}$ is given by
	$$N_{\tilde\f}(X,Y)=(N_\f(X_1,Y_1),N_J(X_2,Y_2))=(2d\eta(X_1,Y_1)\xi,0)=2d\tilde\eta(X,Y)\tilde\xi.$$
	Finally we observe that if $M^{2n+1}$ has rank $4p+1$, then $M^{2n+1}\times M^{2m}$ has the same rank.
	
	Considering the weigthed Heisenberg Lie group $G$ described in Section \ref{Ex:w-Heisenberg}, assuming $\lambda_{p+1}=\cdots=\lambda_n=0$ for some $1\le p\le n-1$, then one can easily see that $G=G'\times\R^{4(n-p)}$, where $G'$ is a weighted Heisenberg Lie group with structure of maximal rank and the abelian Lie group $\R^{4(n-p)}$ is actually endowed with a hyperk\"ahler structure.
	
	In Section \ref{Sec:connection} we will provide a sufficient condition for an anti-quasi-Sasakian manifold to be decomposable as Riemannian product of an anti-quasi-Sasakian manifold of maximal rank and a K\"ahler manifold.
\end{example}

\section{Riemannian curvature properties}\label{Sec:curvature}
In this section we will further investigate the geometric structure of an anti-quasi-Sasakian manifold, showing that it carries a triplet of closed 2-forms $(\A,\Phi,\Psi)$, the third one being exact. This allows to provide a characterization of anti-quasi-Sasakian manifolds of constant $\xi$-sectional curvature equal to $1$, and to obtain remarkable properties of the Riemannian Ricci tensor. Then we show that an aqS manifold of constant sectional curvature is flat and cok\"ahler. Other obstructions to the existence of aqS structures are discussed in the compact and homogeneous cases.

\subsection{The triplet of closed 2-forms $(\A,\Phi,\Psi)$}
Let $(M,\f,\xi,\eta,g)$ be an anti-quasi-Sasakian manifold, and consider the operator $A=(\nabla\f)\xi=-\f\circ\nabla\xi$ defined in Theorem \ref{Thm:char.aqS}. We set
\begin{equation}\label{eq:psi}
	\psi:=A\f=-\f A=-\nabla\xi.
\end{equation}
Being $\xi$ Killing, $\psi$ is skew-symmetric with respect to $g$.
Moreover, by the fact that $A$ anticommutes with $\f$ and $A\xi=0$, the following identities hold:
\begin{equation}\label{eq:fpsi=-psif}
	\f\psi=A=-\psi\f,
\end{equation}
\begin{equation}\label{eq:psi A=-Apsi}
	\psi A=-\f A^2=-A\psi.
\end{equation}
From \eqref{eq:g(AX,Z)} we also get:
\begin{equation}\label{eq:g-psi-deta}
	2g(X,\psi Y)=d\eta(X,Y).
\end{equation}

\begin{remark}\label{Rmk:psi^2}
	Note that $$\psi^2=(A\f)(A\f)=-A^2\f^2=-A^2(-I+\eta\otimes\xi)=A^2.$$
	Moreover, since $\psi$ is skew-symmetric, $\psi^2$ is symmetric with respect to $g$, with nonpositive eigenfunctions.
	In particular, if $X\in\Gamma(\D)$ is an eigenvector field of $\psi^2$ with eigenfunction $-\lambda^2<0$, then $X,\f X,\psi X,AX$ are mutually orthogonal eigenvector fields associated to the same eigenfunction. This follows from equations \eqref{eq:psi}, \eqref{eq:fpsi=-psif}, \eqref{eq:psi A=-Apsi} and from the skew-symmetry of $\f$, $A$, $\psi$ and $\psi A$ with respect to $g$.
	Finally, $\operatorname{Ker}\psi^2=\operatorname{Ker}\psi$ is $\f$-invariant and it coincides with $\left<\xi\right>\oplus\mathcal{E}$ by \eqref{eq:g-psi-deta}, where $\mathcal{E}$ is defined as in Proposition \ref{Prop.rank}. In particular $\psi=0$ if and only if the manifold is cok\"ahler.
\end{remark}
\smallskip

\begin{proposition}\label{Prop:Lie_psi/A}
	Let $(M,\f,\xi,\eta,g)$ be an anti-quasi-Sasakian manifold. Then:
	\begin{equation*}
		\Lie_\xi\psi=0,\quad \Lie_\xi A=0.
	\end{equation*}
\end{proposition}
\begin{proof}
	The first identity is an immediate consequence of the fact that $\xi$ is a Killing vector field, $\Lie_\xi d\eta=0$ by Proposition \ref{prop-deta-normal}, and equation \eqref{eq:g-psi-deta}. The second identity follows from \eqref{eq:fpsi=-psif} and $\Lie_\xi\f=\Lie_\xi\psi=0$.
\end{proof}
\smallskip

\begin{proposition}\label{Prop:nabla_xi}
	Let $(M,\f,\xi,\eta,g)$ be an anti-quasi-Sasakian manifold. Then:
	\begin{equation}\label{3nabla}
		\nabla_\xi\f=2A,\quad \nabla_\xi\psi=0,\quad \nabla_\xi A=-2\psi A.
	\end{equation}
\end{proposition}
\begin{proof}
	The first equation is a direct consequence of \eqref{eq:aqS.char} applied for $X=\xi$. For every $X\in\X(M)$ we have:
	\begin{eqnarray*}
		(\Lie_\xi\psi)X&=&[\xi,\psi X]-\psi[\xi,X]\\
		&=&\nabla_\xi\psi X-\nabla_{\psi X}\xi-\psi\nabla_X\xi+\psi\nabla_X\xi\\
		&=&(\nabla_\xi\psi)X+\psi^2X-\psi^2X=(\nabla_\xi\psi)X,
	\end{eqnarray*}
	and hence the second equation follows from Proposition \ref{Prop:Lie_psi/A}. Finally we have:
	$$\nabla_\xi A=(\nabla_\xi\f)\psi+\f\nabla_\xi\psi=2A\psi=-2\psi A.$$
\end{proof}

\noindent Now, let us consider the $2$-forms associated to the skew-symmetric operators $A$ and $\psi$, defined by:
$$\mathcal{A}(X,Y):=g(X,AY),\quad \Psi(X,Y):=g(X,\psi Y)$$
for every $X,Y\in\X(M)$. Equation \eqref{eq:g-psi-deta} immediately implies $d\Psi=0$. We are going to show that $\A$ is closed as well, and to this aim we prove some useful identities for $\nabla\Psi$.
\smallskip

\begin{lemma}\label{Lemma:identita nablaPsi}
	Let $(M,\f,\xi,\eta,g)$ be an anti-quasi-Sasakian manifold. Then, for every $X,Y,Z\in\X(M)$ the following identities hold:
	\begin{align*}
		(\nabla_X\Psi)(\f Y,Z)-(\nabla_X\Psi)(Y,\f Z)&= \eta(Y)g(\psi^2X,\f Z)+\eta(Z)g(\psi^2X,\f Y),\\
		(\nabla_X\Psi)(\f Y,\f Z)+(\nabla_X\Psi)(Y,Z)&= {}-\eta(Y)g(\psi^2X,Z)+\eta(Z)g(\psi^2X,Y),\\
		(\nabla_{\f X}\Psi)(\f Y,Z)+(\nabla_X\Psi)(Y,Z)&= {}-\eta(Y)g(\psi^2X,Z)+2\eta(Z)g(\psi^2X,Y).
	\end{align*}
\end{lemma}
\begin{proof}
	Recall that $A=\f\psi=-\psi\f$ so that, applying Theorem \ref{Thm:char.aqS} and using $\eta\circ\psi=0$ and $\psi\xi=0$, one obtains:
	\begin{eqnarray}\label{eq:nablaA}
		(\nabla_XA)Y&=&(\nabla_X\f\psi)Y\nonumber\\
		&=&(\nabla_X\f)\psi Y+\f(\nabla_X\psi)Y\nonumber\\
		&=&2\eta(X)A\psi Y+g(X,A\psi Y)\xi+\f(\nabla_X\psi)Y,
	\end{eqnarray}
	and
	\begin{eqnarray*}
		(\nabla_XA)Y&=&-(\nabla_X\psi\f)Y\\
		&=&-(\nabla_X\psi)\f Y-\psi(\nabla_X\f)Y\\
		&=&-(\nabla_X\psi)\f Y-2\eta(X)\psi AY-\eta(Y)\psi AX.
	\end{eqnarray*}
	Therefore, since $\psi$ and $A$ anticommute each other, we get:
	\begin{eqnarray*}
		(\nabla_X\psi)\f Y+\f(\nabla_X\psi)Y&=& \eta(Y)A\psi X-g(X,A\psi Y)\xi\\
		&=&\eta(Y)\f\psi^2X-g(\psi^2X,\f Y)\xi.
	\end{eqnarray*}
	Using $(\nabla_X\Psi)(Y,Z)=-g((\nabla_X\psi)Y,Z)$, the first identity is proved. The second equation follows from the first one just by replacing $Z$ by $\f Z$:
	\begin{eqnarray*}
		(\nabla_X\Psi)(\f Y,\f Z)+(\nabla_X\Psi)(Y,Z)
		&=&\eta(Z)(\nabla_X\Psi)(Y,\xi)-\eta(Y)g(\psi^2X,Z)\\
		&=&\eta(Z)g((\nabla_X\psi)\xi,Y)-\eta(Y)g(\psi^2X,Z)\\
		&=&\eta(Z)g(\psi^2X,Y)-\eta(Y)g(\psi^2X,Z).
	\end{eqnarray*}
	Finally, using $d\Psi=0$ and applying the first two identities, one has:
	\begin{eqnarray*}
		0&=&d\Psi(X,Y,Z)+d\Psi(\f X,\f Y,Z)+d\Psi(\f X,Y,\f Z)-d\Psi(X,\f Y,\f Z)\\
		&=&(\nabla_X\Psi)(Y,Z)+(\nabla_Y\Psi)(Z,X)+(\nabla_Z\Psi)(X,Y)\\
		&&{}+(\nabla_{\f X}\Psi)(\f Y,Z)+(\nabla_{\f Y}\Psi)(Z,\f X)+(\nabla_Z\Psi)(\f X,\f Y)\\
		&&{}+(\nabla_{\f X}\Psi)(Y,\f Z)+(\nabla_Y\Psi)(\f Z,\f X)+(\nabla_{\f Z}\Psi)(\f X,Y)\\
		&&{}-(\nabla_X\Psi)(\f Y,\f Z)-(\nabla_{\f Y}\Psi)(\f Z,X)-(\nabla_{\f Z}\Psi)(X,\f Y)\\
		&=&(\nabla_X\Psi)(Y,Z)+(\nabla_{\f X}\Psi)(\f Y,Z)+(\nabla_{\f X}\Psi)(Y,\f Z)-(\nabla_X\Psi)(\f Y,\f Z)\\
		&&{}-\eta(Z)g(\psi^2Y,X)+\eta(X)g(\psi^2Y,Z)-\eta(X)g(\psi^2Z,Y)+\eta(Y)g(\psi^2Z,X)\\
		&&{}-\eta(Z)g(\psi^2\f Y,\f X)-\eta(X)g(\psi^2\f Y,\f Z)+\eta(X)g(\psi^2\f Z,\f Y)+\eta(Y)g(\psi^2\f Z,\f X)\\
		&=&(\nabla_X\Psi)(Y,Z)+(\nabla_{\f X}\Psi)(\f Y,Z)+(\nabla_{\f X}\Psi)(Y,\f Z)-(\nabla_X\Psi)(\f Y,\f Z)\\
		&&{}-2\eta(Z)g(\psi^2Y,X)+2\eta(Y)g(\psi^2Z,X).
	\end{eqnarray*}
	Then, by adding and subtracting $(\nabla_{\f X}\Psi)(\f Y,Z)+(\nabla_{\f X}\Psi)(Y,Z)$, and applying again the first two identities, one has:
	\begin{eqnarray*}
		\lefteqn{2\{(\nabla_X\Psi)(Y,Z)+(\nabla_{\f X}\Psi)(\f Y,Z)\}}\\
		&=&(\nabla_{\f X}\Psi)(\f Y,Z)-(\nabla_{\f X}\Psi)(Y,\f Z)+(\nabla_X\Psi)(\f Y,\f Z)+(\nabla_X\Psi)(Y,Z)\\
		&&{}+2\eta(Z)g(X,\psi^2 Y)-2\eta(Y)g(X,\psi^2Z)\\
		&=&\eta(Y)g(\psi^2\f X,\f Z)+\eta(Z)g(\psi^2\f X,\f Y) -\eta(Y)g(\psi^2X,Z)+\eta(Z)g(\psi^2X,Y)\\
		&&{}+2\eta(Z)g(\psi^2X,Y)-2\eta(Y)g(\psi^2X,Z)\\
		&=&4\eta(Z)g(\psi^2X,Y)-2\eta(Y)g(\psi^2X,Z),
	\end{eqnarray*}
	and this concludes the proof of the third equation.
\end{proof}
\medskip

\begin{proposition}\label{Prop:dA=0}
	For every anti-quasi-Sasakian manifold $(M,\f,\xi,\eta,g)$ the $2$-form $\A$ is closed.
\end{proposition}
\begin{proof}
	By using $\A(\xi,\cdot)=0$, $\xi$ Killing and $\Lie_\xi A=0$, for every $X,Y\in\X(M)$ we have:
	\begin{eqnarray*}
		d\A(\xi,X,Y)&=&\xi(\A(X,Y))+X(\A(Y,\xi))+Y(\A(\xi,X))\\
		&&-\A([\xi,X],Y)-\A([X,Y],\xi)-\A([Y,\xi],X)\\
		&=&\xi(g(X,AY))-g([\xi,X],AY)-g([Y,\xi],AX)\\
		&=&g(X,[\xi,AY])-g(A[\xi,Y],X)\\
		&=&g(X,(\Lie_\xi A)Y)=0.
	\end{eqnarray*}
	Now, for every $X,Y,Z\in\Gamma(\D)$, equation \eqref{eq:nablaA} gives that $g((\nabla_XA)Y,Z)=(\nabla_X\Psi)(Y,\f Z)$. Thus we have:
	\begin{eqnarray*}
		d\A(X,Y,Z)
		&=&-g((\nabla_XA)Y,Z)-g((\nabla_YA)Z,X)-g((\nabla_ZA)X,Y)\\
		&=&-(\nabla_X\Psi)(Y,\f Z)-(\nabla_Y\Psi)(Z,\f X) -(\nabla_Z\Psi)(X,\f Y)
	\end{eqnarray*}
	Since $d\Psi(X,Y,\f Z)=0$, we get
	$$-(\nabla_X\Psi)(Y,\f Z)=(\nabla_Y\Psi)(\f Z,X)+(\nabla_{\f Z}\Psi)(X,Y).$$
	Substituting in the previous equation and applying the first and the third identities in Lemma \ref{Lemma:identita nablaPsi}, we obtain:
	\begin{eqnarray*}
		d\A(X,Y,Z)
		&=&(\nabla_Y\Psi)(\f Z,X)+(\nabla_{\f Z}\Psi)(X,Y)-(\nabla_Y\Psi)(Z,\f X)-(\nabla_Z\Psi)(X,\f Y)\\
		&=&(\nabla_{\f Z}\Psi)(\f^2Y,X)+(\nabla_Z\Psi)(\f Y,X)=0.
	\end{eqnarray*}
\end{proof}

\subsection{Sectional curvatures and Ricci curvature}
Let us denote by $R$ the Riemannian curvature tensor field of $(M,g)$, defined by $R(X,Y)Z=[\nabla_X,\nabla_Y]Z-\nabla_{[X,Y]}Z$. The sectional curvature at $x\in M$ of a 2-plane spanned by orthonormal $u,v\in T_xM$ is given by $K(u,v)=g_x(R_x(u,v)v,u)$. In the following, for an almost contact metric manifold $(M,\f,\xi,\eta,g)$ and for any unit vector field $X\in\Gamma(\D)$, we will consider the $\xi$-sectional curvature $K(\xi,X)$ defined at each point $x\in M$ by $K(\xi_x,X_x)$. 

\begin{proposition}\label{Prop:nabla_psi}
	Let $(M,\f,\xi,\eta,g)$ be an anti-quasi-Sasakian manifold. Then, for every $X,Y\in\X(M)$ one has:
	\begin{itemize}
		\item[(i)] $(\nabla_X\psi)Y=R(\xi,X)Y$;
		\item[(ii)] $R(X,Y)\xi=-(\nabla_X\psi)Y+(\nabla_Y\psi)X$;
		\item[(iii)] $\psi^2X=R(\xi,X)\xi$.
	\end{itemize}
	In particular, $M$ has nonnegative $\xi$-sectional curvatures, and for every unit eigenvector field $X\in\Gamma(\D)$ of $\psi^2$, with eigenfunction $-\lambda^2$, $K(\xi,X)=\lambda^2$.
\end{proposition}
\begin{proof}
	Since $\xi$ is Killing and $\psi=-\nabla\xi$, the first identity follows from a well known formula (see \cite[vol. I, Chap. VI, Proposition 2.6]{KN}). Then (ii) follows from the Bianchi identity:
	$$R(X,Y)\xi=-R(Y,\xi)X-R(\xi,X)Y=R(\xi,Y)X-R(\xi,X)Y.$$
	Applying (i) for $Y=\xi$ we get (iii):
	$$R(\xi, X)\xi=(\nabla_X\psi)\xi=-\psi\nabla_X\xi=\psi^2X.$$
	Moreover, for a unitary $X\in\Gamma(\D)$, the $\xi$-sectional curvature is:
	\begin{equation}\label{eq:K(X,xi)}
		K(\xi,X)=-g(R(\xi,X)\xi,X)=-g(\psi^2X,X)=g(\psi X,\psi X)\ge0,
	\end{equation}
	which also immediately justifies the last statement.
\end{proof}
\medskip

Notice that the $\xi$-sectional curvatures of an anti-quasi-Sasakian manifold $M$ are all vanishing if and only if $M$ is cok\"ahler.
Next we provide a characterization of anti-quasi-Sasakian manifolds with $K(\xi,X)=1$, showing that they are all double aqS-Sasakian manifolds.

\begin{theorem}\label{Thm:K(X,xi)=1}
	Let $(M,\f,\xi,\eta,g)$ be anti-quasi-Sasakian manifold. Then the following are equivalent:
	\begin{itemize}
		\item[(a)] $K(\xi,X)=1$ for every $X\in\Gamma(\D)$;
		\item[(b)] $\psi^2=A^2=-I+\eta\otimes\xi$;
		\item[(c)] $(A,\f,\psi,\xi,\eta,g)$ is a double aqS-Sasakian structure.
	\end{itemize}
\end{theorem}
\begin{proof}
	Recall that the $(1,1)$-tensor fields $A$ and $\psi$ attached to the anti-quasi-Sasakian structure satisfy $\psi^2=A^2$ and $\psi\xi=A\xi=0$. Thus, they define almost contact structures with respect to $(\xi,\eta)$ if and only if $\psi^2|_\D=-I$, i.e. the spectrum of $\psi^2$ is $\{0,-1\}$, with $0$ simple eigenvalue. By Proposition \ref{Prop:nabla_psi}, this is equivalent to require that the $\xi$-sectional curvatures are equal to $1$.\\
	Assuming $\psi^2=A^2=-I+\eta\otimes\xi$, the compatibility of the almost contact structures $(\psi,\xi,\eta)$ and $(A,\xi,\eta)$ with the Riemannian metric $g$ is consequence of the skew-symmetry of $\psi$ and $A$. Indeed, for every $X,Y\in\X(M)$:
	$$g(\psi X,\psi Y)=-g(\psi^2X,Y)=g(X,Y)-\eta(X)\eta(Y),$$
	and similarly for $A$.
	Finally, \eqref{eq:psi A=-Apsi} becomes $\psi A=\f=-A\psi$ and, together with \eqref{eq:psi} and \eqref{eq:fpsi=-psif}, it implies that $(A,\f,\psi,\xi,\eta,g)$ is an $Sp(n)$-almost contact metric structure.
	Moreover, $d\Phi=0$ (by definition of aqS structure), $d\A=0$ by Proposition \ref{Prop:dA=0}, and $d\eta=2\Psi$ by equation \eqref{eq:g-psi-deta}, so that $(A,\f,\psi,\xi,\eta,g)$ is a double aqS-Sasakian structure.
\end{proof}
\smallskip

\begin{remark}\label{Rmk:scaling}
	Assuming $K(\xi,X)=\lambda^2$ for every $X\in\Gamma(\D)$, or equivalently $\psi^2|_\D=-\lambda^2I$, with $\lambda\in\R\setminus\{0\}$, one can always find an underlying double aqS-Sasakian structure. Indeed, first notice that the class of aqS structures is invariant under homothetic deformations, defined by
	$$\f'=\f,\quad \xi'=\frac1\lambda\xi,\quad \eta'=\lambda\eta,\quad g'=\lambda^2g,$$
	and the associated operators $A'$ and $\psi'$ are given by  $A'=\frac1\lambda A$ and $\psi'=\frac1\lambda\psi$.
	Therefore, if  $\psi^2|_\D=-\lambda^2I$, then $$\psi'^2=A'^2=-I+\eta'\otimes\xi',$$
	and $(A',\f',\psi',\xi',\eta',g')$ is double aqS-Sasakian by the previous theorem.
	In particular, by Theorem \ref{Thm:Sp(n)-struct.}, it locally projects onto a hyperk\"ahler structure. Moreover, $g$ and $g'$ project on homothetic Riemannian metrics along the vertical distribution, so that both $g$ and $g'$ are transversely Ricci-flat (cfr. Theorem \ref{Prop.eta-Einstein}).
\end{remark}
\medskip

\begin{remark}
	The condition $K(\xi,X)=\lambda^2>0$ (or in particular $K(\xi,X)=1$) for every $X\in\Gamma(\D)$ is not always satisfied. For instance, considering the two anti-quasi-Sasakian structures of the weighted Heisenberg Lie group (Section \ref{Ex:w-Heisenberg}), one has:
	$$\operatorname{Sp}(A^2_{\f_1})=\operatorname{Sp}(A^2_{\f_2})=\{0,-\lambda_1^2,\dots,-\lambda_n^2\}.$$
\end{remark}
\medskip

\begin{proposition}\label{Prop:Ric}
	Let $(M,\f,\xi,\eta,g)$ be an anti-quasi-Sasakian manifold. Then the Ricci tensor field satisfies the following identities:
	\begin{itemize}
		\item[(i)] $\Ric(\xi,\xi)=|\nabla\xi|^2=|\psi|^2$;
		\item[(ii)] $\Ric(\xi,X)=0$;
		\item[(iii)] $\Ric(X,Y)=\Ric^T(X',Y')-2g(\psi X,\psi Y)$,
	\end{itemize}
	where  $\Ric^T$ is the Ricci tensor field of the base space of the local Riemannian submersion $\pi:M\to M/\xi$, and $X,Y\in\Gamma(\mathcal{D})$ are basic vector fields projecting on $X',Y'$ respectively.
	Furthermore, the scalar curvatures of $M$ and $M/\xi$ are related by
	$$s=s^T-|\nabla\xi|^2=s^T-|\psi|^2.$$
\end{proposition}
\begin{proof}
	Choosing a local othornormal frame $\{\xi,e_1,\dots,e_{2n}\}$ (where $\dim M=2n+1$), as in \eqref{eq:K(X,xi)} we have that
	$$\Ric(\xi,\xi)=\sum_{i=1}^{2n}K(\xi,e_i)=\sum_{i=1}^{2n}g(\psi e_i,\psi e_i)=|\psi|^2.$$
	
	In order to prove equation (ii), we fix a local orthonormal $\f$-basis, that is $e_{n+i}=\f e_i$ for $i=1,\dots,n$. For every $X\in\Gamma(\D)$, applying property (i) of Proposition \ref{Prop:nabla_psi}, we have:
	\begin{eqnarray*}
		\Ric(\xi,X)&=&\sum_{i=1}^{n}\{g(R(\xi,e_i)e_i,X)+g(R(\xi,\f e_i)\f e_i,X)\}\\
		&=&\sum_{i=1}^{n}\{g((\nabla_{e_i}\psi)e_i,X)+g((\nabla_{\f e_i}\psi)\f e_i,X)\}\\
		&=&{}-\sum_{i=1}^{n}\{(\nabla_{e_i}\Psi)(e_i,X)+(\nabla_{\f e_i}\Psi)(\f e_i,X)\}=0,
	\end{eqnarray*}
	where we applied the third identity of Lemma \ref{Lemma:identita nablaPsi}, being $\eta(e_i)=\eta(X)=0$.
	
	Now, consider the local Riemannian submersion $\pi:M\to M/\xi$ and let $\mathsf{T}$ and $\mathsf{A}$ be the O'Neill tensors related to $\pi$. Since $\nabla_\xi\xi=0$, the leaves of the distribution spanned by $\xi$ are totally geodesic, and hence $\mathsf{T}=0$. Moreover, for every $X,Y$ horizontal vector fields
	$$\mathsf{A}_XY=\frac12v([X,Y])=\frac12\eta([X,Y])\xi=-\frac12d\eta(X,Y)\xi=-\Psi(X,Y)\xi,$$
	where $v$ denotes the vertical component.
	Then, applying equations in \cite[Proposition 1.7]{Falcitelli}, for every $X,Y$ basic vector fields projecting on $X',Y'$ respectively, one has:
	\begin{eqnarray*}
		\Ric(X,Y)&=&\Ric^T(X',Y')-2\sum_{i=1}^{2n}g(\mathsf{A}_Xe_i,\mathsf{A}_Ye_i)\\
		&=&\Ric^T(X',Y')-2\sum_{i=1}^{2n}\Psi(X,e_i)\Psi(Y,e_i)\\
		&=&\Ric^T(X',Y')-2\sum_{i=1}^{2n}g(\psi X,e_i)g(\psi Y,e_i)\\
		&=&\Ric^T(X',Y')-2g(\psi X,\psi Y).
	\end{eqnarray*}
	Finally, applying equations (i) and (iii), we get:
	$$	s=\sum_{i=1}^{2n}\Ric(e_i,e_i)+\Ric(\xi,\xi)
	=\sum_{i=1}^{2n}\Ric^T(e'_i,e'_i)-2\sum_{i=1}^{2n}g(\psi e_i,\psi e_i)+|\psi|^2
	=s^T-|\psi|^2.$$
\end{proof}
\smallskip

\begin{corollary}
	The Reeb vector field of an anti-quasi-Sasakian manifold $(M,\f,\xi,\eta,g)$ is an eigenvector field of the Ricci operator $Q$. Moreover, $Q\f=\f Q$.
\end{corollary}
\begin{proof}
	Recall that the Ricci operator $Q$ is defined by $g(QX,Y)=\Ric(X,Y)$ for every $X,Y\in\X(M)$. Then, the first two identities in the above proposition easily yield $Q\xi=|\psi|^2\xi$.
	Moreover, given two basic vector fields $X,Y\in\Gamma(\D)$ projecting onto $X',Y'\in\X(M/\xi)$, $\f X$ and $\f Y$ project onto $JX'$ and $JY'$ respectively. Using the third equation in Proposition \ref{Prop:Ric}, since the Ricci curvature of a K\"ahler manifold is $J$-invariant and $\psi\f=-\f\psi$, we have:
	\begin{eqnarray*}
		\Ric(\f X,\f Y)&=&\Ric^T(JX',JY')-2g(\psi\f X,\psi\f Y)\\
		&=&\Ric^T(X',Y')-2g(\psi X,\psi Y)\\
		&=&\Ric(X,Y).
	\end{eqnarray*}
	The equality holds true replacing $Y$ by $\xi$, since $\Ric(X,\xi)=0$ for every $X\in\Gamma(\D)$.
	Hence we showed that $-g(\f Q\f X,Y)=g(QX,Y)$ for every $X\in\Gamma(\D)$ and $Y\in\X(M)$, i.e. $\f Q\f=-Q$ along $\D$. This, together with $\f Q\xi=0$, gives $\f Q=Q\f$.
\end{proof}
\medskip

Recall that an almost contact metric manifold is called $\eta$-Einstein if
\begin{equation}\label{eq:def.eta_Einstein}
	\Ric=\mu g+\nu\eta\otimes\eta,\quad \mu,\nu\in\R.
\end{equation}
We investigate this condition for the class of anti-quasi-Sasakian manifolds.

\begin{theorem}\label{Prop.eta-Einstein}
	Let $(M,\f,\xi,\eta,g)$ be a transversely Einstein, non cok\"ahler, anti-quasi-Sasakian manifold. Then it is $\eta$-Einstein if and only if $\psi^2|_\D=-\lambda^2I$, with $\lambda\in\R\setminus\{0\}$. In this case $M$ turns out to be transversely Ricci-flat; the Ricci tensor and the scalar curvature of $g$ are given by
	\begin{equation}\label{eq:Ric_lambda}
		\Ric=-2\lambda^2g+(4n+2)\lambda^2\eta\otimes\eta,\quad s=-4n\lambda^2,
	\end{equation}
	where $\dim M=4n+1$.
\end{theorem}
\begin{proof}
	Since $(M,\f,\xi,\eta,g)$ is transversely Einstein, equation (iii) in Proposition \ref{Prop:Ric} reduces to
	\begin{equation}\label{eq:Ricci-eta-Einstein}
		\Ric(X,Y)=\rho g(X,Y)+2g(X,\psi^2 Y)
	\end{equation} for some $\rho\in\R$ and for every $X,Y\in\Gamma(\D)$ basic vector fields. Taking into account also equations (i) and (ii) of the same proposition, we conclude that $M$ is $\eta$-Einstein if and only if $\psi^2|_\D=-\lambda^2I$, for some constant $\lambda\in\R\setminus\{0\}$. In this case, by Remark \ref{Rmk:scaling}, $M$ is transversely Ricci-flat ($\rho=0$). Since the aqS structure is of maximal rank, $\dim M=4n+1$ and hence $\Ric(\xi,\xi)=4n\lambda^2$. Using also equation \eqref{eq:Ricci-eta-Einstein} with $\rho=0$ and $\Ric(\xi,X)=0$ for every $X\in\Gamma(\D)$, one gets the expressions for the Ricci tensor and the scalar curvature.
\end{proof}
\medskip

\begin{remark}\label{Rmk:Ric_lambda}
	It is worth remarking that, arguing as in the proof of the above theorem, the condition $\psi^2|_\D=-\lambda^2I$, $\lambda\neq0$, together with Remark \ref{Rmk:scaling}, gives the explicit expression of the Ricci tensor field in \eqref{eq:Ric_lambda}, without assuming the metric to be transversely Einstein.
	
	In particular, for anti-quasi-Sasakian manifolds with $K(\xi,X)=1$ (see Theorem \ref{Thm:K(X,xi)=1}), $$\Ric=-2g+(4n+2)\eta\otimes\eta,\quad s=-4n$$ and $(M,\psi,\xi,\eta,g)$ turns out to be a \textit{null Sasakian $\eta$-Einstein manifold} (see \cite{BGM}).
\end{remark}
\medskip

In general the Riemannian metric of an anti-quasi-Sasakian manifold is not necessarily $\eta$-Einstein as in the following two examples.
\medskip

\begin{example}
	The weighted Heisenberg Lie group is transversely flat, and by Proposition \ref{Prop:Ric}, in the fixed orthonormal basis of left invariant vector fields, the Riemannian Ricci tensor is represented by the following diagonal matrix
	$$\begin{pmatrix}
		-8\sum_{i=1}^n\lambda_i^2&&&\\
		&-2\lambda_1^2I_4&&\\
		&&\ddots&\\
		&&&-2\lambda_n^2I_4
	\end{pmatrix}.$$
\end{example}
\smallskip

\begin{example}\label{example-disc}
	Consider the principal $\mathbb{S}^1$-bundle $M$ over the complex unit disc $(D^2,J,k)$, endowed with the aqS structure $(\f,\xi,\eta,g)$ defined as in Example \ref{Ex:disco}. Since $D^2$ has constant holomorphic sectional curvature $c<0$, the manifold $M$ is transversely Einstein (not transversely Ricci-flat). In fact the structure here is not $\eta$-Einstein because $\psi^2$ has a single non constant eigenfunction as shown in the following. In the global frame $\{\frac{\partial}{\partial x^i},\frac{\partial}{\partial y^i}, \frac{d}{dt}\}$, $i=1,2$, the Riemannian metric $g$ and the $2$-form $d\eta=\pi^*\omega$ are represented by the matrices
	$$G=\begin{pmatrix}
		b_2&b_3&0&b_4&0\\
		b_3&b_1+x_1^2&-b_4&-x_1y_1&x_1\\
		0&-b_4&b_2&b_3&0\\
		b_4&-x_1y_1&b_3&b_1+y_1^2&-y_1\\
		0&x_1&0&-y_1&1\\
	\end{pmatrix},\quad
	F=\begin{pmatrix}
		0&1&0&0&0\\
		-1&0&0&0&0\\
		0&0&0&-1&0\\
		0&0&1&0&0\\
		0&0&0&0&0
	\end{pmatrix}$$
	where $$b_1=a(1-x_1^2-y_1^2),\ b_2=a(1-x_2^2-y_2^2),\ b_3=a(x_1x_2+y_1y_2),\ b_4=a(x_1y_2-x_2y_1),$$ with $a=\frac{-4}{c(1-|z|^2)^2}$ and $z=(x_1,x_2,y_1,y_2)$.\\
	Using $d\eta=2g(\cdot,\psi\cdot)$, the matrix of $\psi$ is $P=\frac12G^{-1}F$, and the matrix $P^2$ associated to $\psi^2$ admits a unique strictly negative eigenfunction of multiplicity $4$ given by
	$$-\lambda^2=-\frac{c^2(1-|z|^2)^3}{64}.$$
	Notice that, since $\Ric^T=\frac32c\,k$, from Proposition \ref{Prop:Ric} it follows that $\Ric=\mu g+\nu\eta\otimes\eta$ with $\mu=\frac32c-2\lambda^2$ and $\nu=6\lambda^2-\frac32c$ non constant functions. It is worth remarking that this provides a difference with respect to the class of $K$-contact manifolds. Indeed, for such manifolds, assuming $\dim M\ge5$, if equation \eqref{eq:def.eta_Einstein} is satisfied for some functions $\mu$ and $\nu$, then these are necessarily constant (see \cite[Proposition 11.8.1]{BoyGal}).
\end{example}
\medskip

We apply the above results on the Riemannian Ricci curvature to the case of aqS manifolds with constant sectional curvature.

\begin{theorem}\label{Thm:constant curvature}
	Let $(M,\f,\xi,\eta,g)$ be an anti-quasi-Sasakian manifold with constant sectional curvature $\kappa$. Then $\kappa=0$ and the manifold is cok\"ahler.
\end{theorem}
\begin{proof}
	Since $M$ has constant sectional curvature $\kappa$, then $M$ is Einstein and for every $X,Y,Z\in\X(M)$ $$R(X,Y)Z=\kappa(g(Y,Z)X-g(X,Z)Y).$$
	In particular for $X=Z=\xi$, applying (iii) of Proposition \ref{Prop:nabla_psi}, we get $\psi^2=\kappa(-I+\eta\otimes\xi)$.
	If $\kappa\neq0$, by Remark \ref{Rmk:Ric_lambda}, $M$ turns out to be $\eta$-Einstein, non Einstein, which is a contradiction. Therefore $\kappa=0$, which implies $\psi=0$ and thus the manifold is cok\"ahler.
\end{proof}

\begin{remark}
	In \cite{Olszak-qS}, Z. Olszak proved that for a quasi-Sasakian manifold of constant sectional curvature $\kappa$, then $\kappa\ge0$. In particular, if $\kappa=0$ the manifold is cok\"ahler, while if $\kappa>0$ the quasi-Sasakian structure is obtained by a homothetic deformation of a Sasakian one.
\end{remark}
\medskip

Other obstructions to the existence of anti-quasi-Sasakian structures come from the fact that no holomorphic $(p,0)$-forms ($p>0)$ can exist on compact K\"ahler manifolds with positive definite Ricci tensor field (\cite[Theorem 20.5]{Moroianu}).
	  
\begin{proposition}
	There exist no compact regular, non cok\"ahler,  anti-quasi-Sasakian manifolds with $\Ric>0$, and no compact regular anti-quasi-Sasakian manifolds of maximal rank with $\Ric\ge0$.
\end{proposition}
\begin{proof}
	Assume that $(M,\f,\xi,\eta,g)$ is a compact regular, non cok\"ahler, anti-quasi-Sasakian manifold. Then $M/\xi$ is a compact K\"ahler manifold endowed with a non-vanishing closed (hence holomorphic) $(2,0)$-form, whose Ricci tensor field satisfies
	\begin{equation}\label{eq:RicT>0}
		\Ric^T(X',X')=\Ric(X,X)+2|\psi X|^2
	\end{equation}
	for every $X'\in\X(M/\xi)$ and $X\in\Gamma(\D)$ basic vector field projecting on $X'$ (see Proposition \ref{Prop:Ric}). In both the cases of the statement it turns out that $\Ric^T>0$, thus contradicting the above mentioned result.
\end{proof}

\begin{proposition}
	There exist no connected homogeneous anti-quasi-Sasakian manifolds of maximal rank with $\Ric\ge0$.
\end{proposition}
\begin{proof}
	Assume that $(M,\f,\xi,\eta,g)$ is a connected homogeneous aqS manifold of maximal rank, with $\Ric\ge0$.  Since homogeneous contact manifolds are regular \cite{BW}, the space of leaves $M/\xi$ is a homogeneous K\"ahler manifold with $\Ric^T>0$ by \eqref{eq:RicT>0}. Then Myers' theorem implies the compactness of $M/\xi$, and thus $M/\xi$ cannot admit any holomorphic form of type $(2,0)$. 
\end{proof}

\begin{corollary}
	There exist no connected $\eta$-Einstein homogeneous anti-quasi-Sasakian manifolds of maximal rank, with $\Ric=\mu g+\nu\eta\otimes\eta$ and $\mu\ge0$.
\end{corollary}
\begin{proof}
	In this case $\Ric(X,X)=\mu|X|^2\ge0$ for every $X\in\Gamma(\D)$. Being also $\Ric(\xi,X)=0$ and $\Ric(\xi,\xi)=\mu+\nu=|\psi|^2$, $\Ric\ge0$ and the result follows from the previous proposition.
\end{proof}

\section{The canonical connection}\label{Sec:connection}
In this section we prove the existence of a compatible metric connection with torsion for an aqS manifold, and we deduce some consequences on the Riemannian geometry of the manifold.\\

Recall that for every Riemannian manifold $(M,g)$, the difference $\bar\nabla-\nabla$ between any linear connection on $M$ and the Levi-Civita connection, is a $(1,2)$-tensor field $H$, which is related to the torsion $\bar{T}$ of $\bar\nabla$ by
\begin{equation}\label{eq:T-H}
	\bar{T}(X,Y)=H(X,Y)-H(Y,X)
\end{equation}
for every $X,Y\in\X(M)$. Let us denote with the same symbols the $(0,3)$-tensor fields derived from $H$ and $\bar{T}$ by contraction with the metric:
$$H(X,Y,Z)=g(H(X,Y),Z),\quad \bar{T}(X,Y,Z)=g(\bar{T}(X,Y),Z).$$
Recall also that $\bar\nabla$ is a metric connection (i.e. $\bar\nabla g=0$) if and only if
\begin{equation}\label{eq:H_metric conn.}
	H(X,Y,Z)+H(X,Z,Y)=0,
\end{equation}
in which case $\bar\nabla$ is completely determined by its torsion, by means of
$$2H(X,Y,Z)=\bar{T}(X,Y,Z)-\bar{T}(Y,Z,X)+\bar{T}(Z,X,Y).$$
A metric connection $\bar\nabla$ is said to have totally skew-symmetric torsion if $\bar{T}\in\Lambda^3(M)$, in which case $2H=\bar{T}$.
We refer the reader to \cite{Agricola_Srni} for a more complete treatment of the theory of connections with torsion.

In \cite{Fr.Iv.}, T. Friedrich and S. Ivanov proved that necessary and sufficient conditions for an almost contact metric manifold $(M,\f,\xi,\eta,g)$ to admit a metric connection $\bar\nabla$ with totally skew-symmetric torsion preserving the almost contact structure (i.e. such that $\bar\nabla\f=0$ and $\bar\nabla\xi=0$), are $\xi$ Killing and $N_\f\in\Lambda^3(M)$. This applies to a large class of almost contact metric manifolds, including quasi-Sasakian manifolds (since $N_\f=0$), but not to anti-quasi-Sasakian manifolds. Indeed, in this case we have that $N_\f=2d\eta\otimes\xi$, namely
$$N_\f(X,Y,Z)=2d\eta(X,Y)\eta(Z)=4g(X,\psi Y)\eta(Z).$$
Thus it can be easily seen that $N_\f$ is not a 3-form, unless $d\eta=0$. Hence, we conclude that an aqS manifold cannot admit a metric connection $\bar\nabla$ with totally skew-symmetric torsion and preserving the structure, unless it is cok\"ahler, in which case $\bar\nabla$ coincides with the Levi-Civita connection. However, we have the following:

\begin{theorem}
	Let $(M,\f,\xi,\eta,g)$ be an anti-quasi-Sasakian manifold. Then, there exists a metric connection $\bar\nabla$ such that $\bar\nabla\f=0$, $\bar\nabla\xi=0$, and the torsion $\bar{T}$ is totally skew-symmetric on $\D$ and satisfies $\bar{T}(\xi,\cdot)=0$.
	The connection $\bar\nabla$ is uniquely determined by $\bar\nabla=\nabla+H$, where
	\begin{equation}\label{eq:H}
		H(X,Y)=\eta(X)\psi Y+\eta(Y)\psi X+g(X,\psi Y)\xi,
	\end{equation}
	and its torsion is given by
	\begin{equation}\label{eq:torsion}
		\bar{T}(X,Y)=2g(X,\psi Y)\xi=d\eta(X,Y)\xi.
	\end{equation}
\end{theorem}
\begin{proof}
	Let $\bar\nabla$ be a metric connection with torsion $\bar{T}$ as in the statement, and let $H:=\bar\nabla-\nabla$. Equation \eqref{eq:T-H} and conditions $\bar{T}(\xi,\cdot)=0$ and $\bar\nabla\xi=0$ imply that $H(\xi,X)=H(X,\xi)=\psi X$ for every $X\in\X(M)$. Thus, since $\bar{\nabla}g=0$ (i.e. \eqref{eq:H_metric conn.} holds), one has
	$$H(X,Y,\xi)=-H(X,\xi,Y)=-H(\xi,X,Y)=-g(\psi X,Y)=\frac12d\eta(X,Y).$$
	Moreover, $$(\bar\nabla_X\f)Y=(\nabla_X\f)Y+H(X,\f Y)-\f H(X,Y),$$
	and, being $\bar\nabla\f=0$, for every $X,Y,Z\in\X(M)$ we have:
	$$H(X,\f Y,Z)+H(X,Y,\f Z)=-g((\nabla_X\f)Y,Z).$$
	In particular, for $X,Y,Z\in\Gamma(\D)$, by equation \eqref{eq:aqS.char} we get:
	\begin{equation}\label{eq:H(.f.)}
		H(X,\f Y,Z)+H(X,Y,\f Z)=0.
	\end{equation}
	Since $\bar{T}$ is totally skew-symmetric on $\D$, $H$ is a 3-form on $\D$. Then, taking the cyclic sum of \eqref{eq:H(.f.)} over $X,Y,Z\in\Gamma(\D)$, one has:
	\begin{eqnarray*}
		0&=&\mathfrak{S}_{XYZ}\{H(X,\f Y,Z)+H(X,Y,\f Z)\}\\
		&=&2\{H(X,Y,\f Z)+H(Y,Z,\f X)+H(Z,X,\f Y)\}.
	\end{eqnarray*}
	Replacing $Y$ by $\f Y$ and using again \eqref{eq:H(.f.)}:
	$$H(X,\f Y,\f Z)+H(\f Y,Z,\f X)-H(Z,X,Y)=H(X,Y,Z)=0$$
	for every $X,Y,Z\in\Gamma(\D)$. Therefore, together with $H(X,Y,\xi)=\frac12d\eta(X,Y)$ we obtain
	$$H(X,Y)=\frac12d\eta(X,Y)\xi\quad \forall X,Y\in\Gamma(\D).$$
	Finally, since $H(\xi,\xi)=0$, for any $X,Y\in\X(M)$
	\begin{eqnarray*}
		H(X,Y)&=&H(\f^2X,\f^2Y)-\eta(X)H(\xi,\f^2Y)-\eta(Y)H(\f^2X,\xi)\\
		&=&\frac12d\eta(\f^2X,\f^2Y)\xi-\eta(X)\psi\f^2Y-\eta(Y)\psi\f^2X\\
		&=&g(X,\psi Y)\xi+\eta(X)\psi Y+\eta(Y)\psi X,
	\end{eqnarray*}
	which proves \eqref{eq:H}. From this and equation \eqref{eq:T-H}, \eqref{eq:torsion} follows.
\end{proof}

We will call the above connection $\bar\nabla$ the \textit{canonical connection} of the aqS structure $(\f,\xi,\eta,g)$. Notice that if the structure has maximal rank, $\bar\nabla$ coincides with the canonical connection defined in \cite{Dileo-Lotta} on a contact manifold $(M,\eta)$ with admissible metric $g$, for which the Reeb vector field $\xi$ is Killing.

In the following we investigate some geometric properties of aqS manifolds under the additional assumption that $\bar\nabla$ has parallel torsion, i.e. $\bar\nabla\bar{T}=0$. In view of \eqref{eq:torsion}, this is equivalent to $\bar\nabla\psi=0$, in which case $\bar\nabla$ also parallelizes $A=\f\psi$. Now, notice that
\begin{eqnarray*}
	(\bar\nabla_X\psi)Y&=&(\nabla_X\psi)Y+H(X,\psi Y)-\psi H(X,Y)\\
	&=&(\nabla_X\psi)Y+g(X,\psi^2Y)\xi-\eta(Y)\psi^2X,
\end{eqnarray*}
and thus $\bar\nabla\psi=0$ if and only if
\begin{equation}\label{eq:nablapsi}
	(\nabla_X\psi)Y=-g(X,\psi^2Y)\xi+\eta(Y)\psi^2X.
\end{equation}

There are two special cases in which $\bar\nabla\psi=0$.
\begin{enumerate}[leftmargin=*]
	\item
For a double aqS-Sasakian manifold $(M,\f_i,\xi,\eta,g)$ ($i=1,2,3$) the canonical connections associated to the two aqS structures $(\f_1,\xi,\eta,g)$ and $(\f_2,\xi,\eta,g)$ coincide. This is the metric connection $\bar\nabla$ with torsion $\bar{T}(X,Y)=d\eta(X,Y)\xi$, which obviously parallelizes $\psi=\f_3=\f_1\f_2$. In fact $\bar\nabla$ is the Tanaka-Webster connection of the Sasakian structure $(\f_3,\xi,\eta,g)$.
	
	\item In the weighted Heisenberg Lie group $G$ discussed in Section \ref{Ex:w-Heisenberg} the canonical connections associated to the aqS structures $(\f_1,\xi,\eta,g)$ and $(\f_2,\xi,\eta,g)$ coincide with the connection $\bar\nabla$ such that
	$\bar\nabla_XY=0$ for every $X,Y\in\mathfrak{g}$. Indeed, one can easily see that $\bar\nabla$ parallelizes all the structure tensor fields and the torsion is given by $$\bar{T}(X,Y)=-[X,Y]=-\eta([X,Y])\xi=d\eta(X,Y)\xi$$
	for every $X,Y\in\mathfrak{g}$. In particular, since $\bar\nabla\bar{T}=0$, one has $\bar\nabla\psi=0$.
\end{enumerate}
\medskip

Now, given an anti-quasi-Sasakian manifold $(M,\f,\xi,\eta,g)$ of rank $4p+1$, let us recall that $\dim M=2q+4p+1$ and the tangent bundle $TM$ splits into the orthogonal sum of two $\f$-invariant distributions $$TM=\mathcal{E}^{2q}\oplus \mathcal{E}^{4p+1},$$
where $\mathcal{E}^{2q}=\operatorname{Ker}(\psi|_\D)$. Next we show that $\bar\nabla\psi=0$ is a sufficient condition for the manifold to be locally decomposable as a Riemannian product.

\begin{theorem}\label{Thm:decomposition}
	Let $(M,\f,\xi,\eta,g)$ be an anti-quasi-Sasakian manifold with  $\bar\nabla\psi=0$. Then:
	\begin{itemize}
		\item[(i)] the distributions $\mathcal{E}^{2q},\ \mathcal{E}^{4p+1}$
		are integrable with totally geodesic leaves and $M$ is locally isometric to a Riemannian product $N^{2q}\times M^{4p+1}$ of a K\"ahler manifold  and an anti-quasi-Sasakian manifold of maximal rank;
		\item[(ii)] if $M$ is connected, then $\psi^2$ has constant eigenvalues. Denoting by $\D_{\mu}$ the eigendistribution of $\psi^2$ associated to a nonvanishing eigenvalue $\mu$, $\left<\xi\right>\oplus\D_{\mu}$ is integrable with totally geodesic leaves. Moreover, every leaf is endowed with a double aqS-Sasakian structure, up to a homothetic deformation.
	\end{itemize}
\end{theorem}
\begin{proof}
	First we show that for all $X\in\X(M)$ and $Y\in\Gamma(\mathcal{E}^{2q})$, $\nabla_XY\in\Gamma(\mathcal{E}^{2q})$. Indeed, being $\psi Y=0$ and $\eta(Y)=0$, we have:
	$$\eta(\nabla_XY)=X(\eta(Y))-g(\nabla_X\xi,Y)=g(\psi X,Y)=-g(X,\psi Y)=0;$$
	furthermore, since $\bar\nabla\psi=0$, applying \eqref{eq:nablapsi}, we get:
	$$\psi\nabla_XY=-(\nabla_X\psi)Y+\nabla_X\psi Y=g(X,\psi^2Y)\xi -\eta(Y)\psi^2X=0.$$
	In particular $\mathcal{E}^{2q}$ turns out to be integrable with totally geodesic leaves. The same holds for $\mathcal{E}^{4p+1}$ since for every $X,Z\in\Gamma(\mathcal{E}^{4p+1})$ and $Y\in\Gamma(\mathcal{E}^{2q})$, one has:
	$$g(\nabla_XZ,Y)=X(g(Z,Y))-g(Z,\nabla_XY)=0,$$
	where each term vanishes because of the orthogonality of the distributions $\mathcal{E}^{2q}$ and $\mathcal{E}^{4p+1}$. Therefore $M$ turns out to be locally isometric to the Riemannian product of a K\"ahler manifold $N^{2q}$ tangent to $\mathcal{E}^{2q}$, and an aqS manifold of maximal rank $M^{4p+1}$, tangent to $\mathcal{E}^{4p+1}$.
	
	Concerning (ii), $\bar\nabla\psi=0$ implies that the eigenvalues of $\psi^2$ are constant and $\left<\xi\right>\oplus\D_\mu$ is a $\bar\nabla$-parallel distribution. Moreover, being $\nabla=\bar\nabla-H$, with $H$ as in $\eqref{eq:H}$, one immediately gets that $\left<\xi\right>\oplus\D_\mu$ is integrable with totally geodesic leaves. Finally, the fact that the leaves of $\left<\xi\right>\oplus\D_{\mu}$ are endowed with a double aqS-Sasakian structure is consequence of Remark \ref{Rmk:scaling}. 
\end{proof}
\smallskip

\begin{remark}
	In \cite{Kanemaki1,Kanemaki2} a similar decomposition theorem has been proved in the quasi-Sasakian case. We already mentioned that for a quasi-Sasakian manifold $(M^{2n+1}\f,\xi,\eta,g)$ the covariant derivative of $\f$ is expressed, by means of equation \eqref{qS}, in terms of a $(1,1)$-tensor field $A=-\f\circ\nabla\xi+k\eta\otimes\xi$. For $k=1$, if $A$ is parallel with respect to the Levi-Civita connection and has constant rank $2p+1$ ($1\le p\le n-1$), then $M$ is locally a product of a K\"ahler manifold $N^{2q}$ and a Sasakian manifold $M^{2p+1}$.

	For an anti-quasi-Sasakian manifold, requiring the parallelism of $\psi$ or $A$ with respect to the Levi-Civita connection implies that $\psi=A=0$, i.e. the manifold is cok\"ahler. Indeed, assuming $\nabla\psi=0$, applying the first equation in Lemma \ref{Lemma:identita nablaPsi} for $Y=\xi$, one has that for every $X,Z\in\X(M)$, $g(\psi^2X,\f Z)=0$ which means $\psi^2=0$ and hence $\psi=0$. Assuming $\nabla A=0$, from the third equation in \eqref{3nabla}, $\psi A=0$, and by \eqref{eq:psi A=-Apsi}, $A^2\varphi=0$, which gives $A=0$.
\end{remark}
\medskip

We prove that for an anti-quasi-Sasakian manifold the condition $\bar\nabla\psi=0$ is not compatible with local Riemannian symmetry.

\begin{theorem}
	There exist no connected, locally symmetric, non cok\"ahler, anti-quasi-Sasakian manifolds with $\bar\nabla\psi=0$.
\end{theorem}
\begin{proof}
	First we point out that a double aqS-Sasakian manifold cannot be locally symmetric. Indeed, a general result due to M. Okumura \cite{Okumura}, states that every locally symmetric Sasakian manifold has constant sectional curvature $1$. By Theorem \ref{Thm:constant curvature}, this is not compatible with the existence of the other two aqS structures in a double aqS-Sasakian manifold.
	
	Now, let $(M,\f,\xi,\eta,g)$ be a non cok\"ahler aqS manifold with $\bar\nabla\psi=0$, and assume that $(M,g)$ is a locally symmetric Riemannian space, i.e. $\nabla R=0$. Let $-\lambda^2$ be a non zero eigenvalue of $\psi^2$. By Theorem \ref{Thm:decomposition}, let $N$ be a maximal integral submanifold of $\langle\xi\rangle\oplus\D_{-\lambda^2}$. $N$ is totally geodesic and hence locally symmetric. Moreover, up to a homothetic deformation of the structure which preserves the local symmetry, $N$ is endowed with a double aqS-Sasakian structure, which is impossible in view of the initial remark.  
\end{proof}

A similar argument shows that the condition $\bar\nabla\psi=0$ is also not compatible with nonnegative sectional curvatures.
	
\begin{proposition}
	There exist no connected, non cok\"ahler, anti-quasi-Sasakian manifolds with nonnegative sectional curvature and such that $\bar\nabla\psi=0$.
\end{proposition}
\begin{proof}
	Assume that $\psi^2$ has a non zero eigenvalue $-\lambda^2$, and consider a maximal integral submanifold $N$ of the distribution  $\langle\xi\rangle\oplus\D_{-\lambda^2}$. Being totally geodesic, $N$ has nonnegative sectional curvature, and hence nonnegative scalar curvature. On the other hand, $N$ is endowed with an aqS structure satisfying $\psi^2|_{\D_{-\lambda^2}}=-\lambda^2I$, and hence $N$ has negative scalar curvature according to \eqref{eq:Ric_lambda}.
\end{proof}

\end{document}